\documentclass[11pt]{amsart}
\usepackage{amscd,amsmath,amssymb,amsfonts,verbatim}
\usepackage{graphicx}

\usepackage[cmtip, all]{xy}

\newtheorem{thm}{Theorem}[subsection]
\newtheorem{prop}[thm]{Proposition}
\newtheorem{lem}[thm]{Lemma}
\newtheorem{cor}[thm]{Corollary}

\theoremstyle{definition}
\newtheorem{exm}[thm]{Example}
\newtheorem{defn}[thm]{Definition}

\theoremstyle{remark}
\newtheorem{remk}[thm]{Remark}
\newtheorem{remks}[thm]{Remarks}
\newtheorem{exms}[thm]{Examples}
\newtheorem{notat}[thm]{Notation}
\newtheorem*{ack}{Acknowledgments}

\numberwithin{equation}{subsection}

{\hfill$\square$\end{defn}}

{\hfill$\square$\end{remk}}

{\hfill$\square$\end{remks}}

{\hfill$\square$\end{exm}}

{\hfill$\square$\end{exms}}

{\hfill$\square$\end{notat}}

\newcommand{\Hom}{{\rm Hom}}

\newcommand{\Spec}{{\rm Spec \,}}

\newcommand{\Sch}{{\operatorname{\mathbf{Sch}}}} 
\newcommand{\Var}{{\operatorname{\mathbf{Var}}}}

\newcommand{\op}{{\text{\rm op}}}

\newcommand{\Spt}{{\mathbf{Spt}}}
\newcommand{\Spc}{{\mathbf{Spc}}}

\newcommand{\Sm}{{\mathbf{Sm}}}

\newcommand{\cofib}{{\operatorname{\rm cofib}}}

\newcommand{\diag}{\mathop{{\rm diag}}}

\newcommand{\Nis}{{\operatorname{Nis}}}

\newcommand{\ds}{{/\kern-3pt/}}

\newcommand{\tr}{{\operatorname{tr}}}

\newcommand{\un}{\underline}
\newcommand{\ov}{\overline}
\newcommand{\wt}{\widetilde}

\newcommand{\tuborg}{\left\{\begin{array}{ll}}
\newcommand{\sluttuborg}{\end{array}\right.}

\newcommand{\sC}{{\mathcal C}}

\newcommand{\sE}{{\mathcal E}}
\newcommand{\sF}{{\mathcal F}}

\newcommand{\sH}{{\mathcal H}}

\newcommand{\sK}{{\mathcal K}}

\newcommand{\sO}{{\mathcal O}}

\newcommand{\sS}{{\mathcal S}}

\newcommand{\A}{{\mathbb A}}

\newcommand{\C}{{\mathbb C}}

\newcommand{\G}{{\mathbb G}}

\newcommand{\bL}{{\mathbb L}}

\renewcommand{\P}{{\mathbb P}}
\newcommand{\Q}{{\mathbb Q}}

\newcommand{\Z}{{\mathbb Z}}

\renewcommand{\top}{{\rm top}}
\newcommand{\alg}{{\rm alg}}

\newcommand{\equi}{{\rm equi}}

\providecommand{\MGL}{\mathop{\rm MGL}\nolimits}

\providecommand{\MU}{\mathop{\rm MU}\nolimits}
\providecommand{\Set}{\mathop{{\bf Set}}\nolimits}

\providecommand{\Spc}{\mathop{{\bf Spc}}\nolimits}
\providecommand{\Hom}{\mathop{\rm Hom}\nolimits}
\providecommand{\sst}{\mathop{{\bf sst}}\nolimits}
\providecommand{\hosst}{\mathop{{\bf host}}\nolimits}

\providecommand{\MGL}{\mathop{\rm MGL}\nolimits}
\providecommand{\colim}{\mathop{\rm colim}}

\begin{document}
\title[Semi-topologization in motivic homotopy theory]
{Semi-topologization in motivic homotopy theory and applications}
\author{Amalendu KRISHNA and Jinhyun PARK}
\address{School of Mathematics, TIFR, Homi Bhabha Road, Mumbai, 400 005, India}
\email{amal@math.tifr.res.in}
\address{Dept. of Math. Sci. KAIST, 291 Daehak-ro, Yuseong-gu,
Daejeon, 305-701, Korea}
\email{jinhyun@mathsci.kaist.ac.kr; jinhyun@kaist.edu}


\begin{abstract}
We study the semi-topologization functor of Friedlander-Walker from the perspective of  motivic homotopy theory. We construct a triangulated endo-functor on the stable motivic homotopy category $\mathcal{SH}(\C)$, which we call \emph{homotopy semi-topologization}. As applications, we discuss the representability of several semi-topological cohomology theories in $\mathcal{SH}(\C)$, a construction of a semi-topological analogue of algebraic cobordism, and a construction of Atiyah-Hirzebruch type spectral sequences for this theory.
\end{abstract}


\maketitle


\section{Introduction}\label{section:Intro}

The goal of this paper is to study semi-topological cohomology theories such as semi-topological $K$-theory of \cite{FW1} and morphic cohomology of \cite{FL} from the perspective of motivic homotopy theory. One feature of the semi-topological theories is that they can be  obtained as \emph{semi-topologizations} of other theories, such as motivic cohomology or algebraic $K$-theory, as pioneered by Friedlander and Walker \cite{FW0}, but semi-topologization does not respect all motivic weak-equivalences, so that it is not an endo-functor on motivic homotopy categories. Nonetheless, we show that it induces a derived functor, call it \emph{homotopy semi-topologization}, using that semi-topologization does respect at least \emph{object-wise} weak-equivalences (see \S\ref{subsec:sst smooth}). So, we first ask when a motivic weak-equivalence may become an object-wise one. After a review on motivic homotopy theory in \S \ref{sec:mot homotopy}, we answer it in \S \ref{sec:A1BG}:

\begin{thm}\label{thm:Main-1}
A motivic weak-equivalence $E \to F$ of $\A^1$-B.G. presheaves on $\Sm_S$ is an object-wise weak-equivalence. Let $T= (\mathbb{P}^1, \infty)$. A stable motivic weak-equivalence $E \to F$ of $\mathbb{A}^1$-B.G. motivic $\Omega_T$-bispectra on $\Sm_S$ is a $T$-level-wise object-wise weak-equivalence. 
\end{thm}

In \S \ref{sec:sst} and \S \ref{sec:sst smooth}, we discuss that these objects where motivic weak-equivalences behave well are closed under the semi-topologization, and we define in \S \ref{sec:sst BG} the \emph{derived} functor on the stable motivic homotopy category $\mathcal{SH}(\C)$:

\begin{thm}\label{thm:Main-2}
There is a triangulated endo-functor $\hosst : \mathcal{SH} (\mathbb{C}) \to \mathcal{SH} (\mathbb{C})$, which coincides with Friedlander-Walker semi-topologization on $\mathbb{A}^1$-B.G. motivic $\Omega_T$-bispectra. 
\end{thm}

Using $\hosst$, in \S \ref{sec:representing-K} and \S \ref{sec:representing-M} we prove the representability of the semi-topological $K$-theory and the morphic cohomology in $\mathcal{SH} (\C)$. In \S \ref{sec:MGL} we define a semi-topological analogue of the algebraic cobordism of \cite{Voevodsky} by simply homotopy semi-topologizing $\MGL$:  

\begin{thm}
\label{thm:Main-3}
The semi-topological $K$-theory and the morphic cohomology are representable in $\mathcal{SH} (\mathbb{C})$. There is a semi-topological cobordism $\MGL_{\sst}$ as a cohomology theory on $\Sm_{\C}$, with a natural transformation $\MGL \to \MGL_{\sst}$ that becomes an isomorphism with finite coefficients. For $X \in \Sm_{\C}$ and $n \ge 0$, there is a spectral sequence $
E_2 ^{p,q} (n) = L^{n-q} H^{p-q} (X) \otimes_{\mathbb{Z}} \mathbb{L}^q \Rightarrow \MGL_{\sst} ^{p+q,n} (X).$ This spectral sequence degenerates after tensoring with $\mathbb{Q}$. 
\end{thm}

\emph{Conventions and notations:} When $S$ is a noetherian scheme of finite Krull dimension, an $S$-scheme is a separated scheme of finite type over $S$. The category of $S$-schemes is $\Sch_S$, while its subcategory of smooth schemes is $\Sm_S$.  
A variety over $k$ is a reduced $k$-scheme, not necessarily quasi-projective. The category of $k$-varieties is $\Var_k$. 

Let $\Set$, $\Spc$, and $\Spc_{\bullet}$ be the categories of sets, simplicial sets, and pointed simplicial sets. Let $\Spt$ be the category of Bousfield-Friedlander spectra \cite{BF} (see \S \ref{sec:S1 stable motivic}). 
The set of maps $K \to L$ in $\Spc_{\bullet}$ is $\Hom_{\bullet}(K,L)$. 

The symbol $\Delta$ is used for the following, and no confusion should arise. First, $\Delta$ is the category whose objects are $[n]:= \{ 0, \cdots, n \}$ for $n \geq 0$ and the morphisms are nondecreasing set functions. The notation $\Delta[n]$ is the simplicial set $\Hom_{\Set} ( - , [n])$ by Yoneda. The notation $\Delta^n _{\top}$ is the topological $n$-simplex $\{ (t_0, \cdots, t_n ) \in \mathbb{R}^{n+1}| 0 \leq t_i \leq 1, \sum_i t_i = 1 \}$, while $\Delta^n$ is the algebraic $n$-simplex $\Spec (k[t_0, \cdots, t_n]/ \sum_i t_i -1)$.

\section{Recollection of motivic homotopy theory}\label{sec:mot homotopy}

We review basics of motivic homotopy theory from \cite{Jardine}, \cite{MV}, and \cite{Morel}. Throughout \S \ref{sec:mot homotopy}-\ref{sec:A1BG}, let $S$ be a fixed noetherian scheme of finite Krull dimension.

\subsection{Motivic spaces} We regard an object of $\Spc$ as a \emph{space}, that of $\Spc_{\bullet}$ as a \emph{pointed space}. A \emph{motivic space over $S$} is a simplicial presheaf on $\Sm_S$. A \emph{pointed motivic space} is a pointed simplicial presheaf on $\Sm_S$. Let $\Spc (S)$ and $\Spc_{\bullet} (S)$ be the categories of unpointed and pointed motivic spaces. A presheaf of sets on $\Sm_S$ is a motivic space of simplicial dimension zero. Each $X \in \Sm_S$ is a motivic space by Yoneda embedding. By $X_+$, we mean $X \coprod S \in \Spc_{\bullet} (S)$. Each (pointed) space $K$ is a (pointed) motivic space, being a constant presheaf on $\Sm_S$. For $U \in \Spc_{\bullet}(S)$, the suspension $\Sigma_U : \Spc_{\bullet} (S) \to \Spc_{\bullet} (S)$ sends $E$ to $ E \wedge U$. For $U = S^1, (\G_m, \{1\})$ and $T = (\mathbb{P}^1, \infty)$, we write $\Sigma_U$ as $\Sigma_s, \Sigma_t$, and $\Sigma_T$. For $E, F$ in $\Spc(S)$ and in $ \Spc_{\bullet} (S)$, let ${\mathcal{H}om}(E, F)$ and ${\mathcal{H}om}_{\bullet} (E,F)$ be the internal hom presheaves of objects in $\Spc$ and $\Spc_{\bullet}$. 
For $E \in \Spc_{\bullet} (S)$, the functor ${\mathcal{H}om}_{\bullet}(E, -)$ on $\Spc_{\bullet}$ is denoted by $\Omega_E(-)$. For $E = S^1$ and $(\G_m, 1)$, we write $\Omega_E(-)$ as $\Omega_s(-)$ and $\Omega_t(-)$.

\subsection{$S^1$-stable motivic homotopy category} \label{sec:S1 stable motivic}
Recall (\cite[Theorem~1.1]{Jardine}) that $\Spc(S)$ is a proper simplicial cellular closed model category, where a map $f: E \to F$ is a \emph{Nisnevich local weak-equivalence} if all induced Nisnevich stalk maps $E_x \to F_x$ are weak-equivalences of $\Spc$, while \emph{cofibrations} are monomorphisms, and \emph{Nisnevich fibrations} are defined in terms of the right lifting property with respect to all trivial cofibrations. A similar model structure on $\Spc_{\bullet} (S)$ exists. Inverting the Nisnevich local weak-equivalences, we get the homotopy categories $\mathcal{H}^{\Nis} (S)$ and $\mathcal{H}^{\Nis} _{\bullet} (S)$. For $E, F \in \Spc_{\bullet} (S)$, let $[E,F]_{\Nis}:=\Hom_{\mathcal{H}^{\Nis} _{\bullet} (S)} (E, F)$. See \cite{Jardine} and \cite{MV} for more details.

A \emph{spectrum}, or an \emph{$S^1$-spectrum}, is a sequence $(E_0, E_1, \cdots)$, $E_i \in \Spc_{\bullet}$, with morphisms $S^1 \wedge E_n \to E_{n+1}$ in $\Spc_{\bullet}$. The category of spectra is $\Spt$, and the category of presheaves of spectra on $\Sm_S$ is $\Spt (S)$. An object of $\Spt(S)$ is called a \emph{motivic spectrum}.

\subsubsection{Nisnevich model structure on motivic spectra over $S$}\label{subsubsection:NMSpec}
A morphism $f:E \to F$ in $\Spt(S)$ is an \emph{object-wise weak-equivalence}, if for each $U \in \Sm_S$, the map $f(U) : E(U) \to F(U)$ is an $S^1$-stable weak-equivalence in $\Spt$. A morphism $f: E \to F$ in $\Spt(S)$ is a \emph{Nisnevich local weak-equivalence}, if for each $U \in \Sm_S$ and $x \in U$, the induced map $f_x : E _x \to F_x$ on the Nisnevich stalks is an $S^1$-stable weak-equivalence in $\Spt$. A map $f: E \to F$ in $\Spt(S)$ is a \emph{cofibration} if $f_0$ is a monomorphism and $E_{n+1} \coprod_{ S^1 \wedge E_n} S^1 \wedge F_n \to F_{n+1}$ is a monomorphism in $\Spc(S)$ for each $n \ge 0$. Equivalently, the maps $E_n \to F_n$ and $S^1 \wedge ({F_n}/{E_n}) \to {F_{n+1}}/{E_{n+1}}$ are monomorphisms in $\Spc(S)$ for each $n \ge 0$. A \emph{Nisnevich fibration} in $\Spt(S)$ is a map with the right lifting property with respect to all trivial cofibrations. Giving a cofibration $E \to F$ in $\Spt(S)$ is equal to giving cofibrations $E(U) \to F(U)$ in $\Spt$. A morphism $E \to F$ in $\Spt(S)$ is a Nisnevich local weak-equivalence if and only if the induced map of Nisnevich sheaves associated to the presheaves $U \mapsto \pi_n (E(U)), \pi_n (F(U))$, is an isomorphism for all $n \in \mathbb{Z}$. Recall:

\begin{thm}[{\cite[Theorem~2.34]{Jardine1}, \cite[Lemma~2.3.6]{Morel}}] \label{thm:MSSpectra}
The above Nisnevich local weak-equivalences, cofibrations and Nisnevich fibrations define a proper simplicial closed model structure on $\Spt (S)$. An object $E$ is cofibrant if and only if the maps $S^1 \wedge E_n \to E_{n+1}$ are monomorphisms. An object $E$ is Nisnevich fibrant if and only if each $E_n$ is a Nisnevich fibrant pointed motivic space and the adjoint  maps $E_n \to \Omega_s^1 E_{n+1}$ are Nisnevich local weak-equivalences.   
\end{thm}

For each $E \in \Spc_{\bullet}(S)$, the motivic spectrum $\Sigma^{\infty}_sE = (E, \Sigma^1_s E, \Sigma^2_s E, \cdots )$ is cofibrant. The homotopy category with respect to the above \emph{Nisnevich local injective model structure} is $\mathcal{SH} _{S^1} ^{\Nis} (S)$. For $E, F \in \Spt(S)$, let $[E,F]_{\Nis}:= \Hom_{\mathcal{SH}_{S^1} ^{\Nis} (S)} (E, F)$.

\subsubsection{Motivic model structure on motivic spectra over $S$}\label{subsubsection:MMSpec}
The homotopy category with respect to the motivic model structure (see \cite{Jardine}) on $\Spc_{\bullet} (S)$ is denoted by $\mathcal{H}_{\bullet} (S)$, and we let $[E, F]_{\A^1}:=\Hom_{\mathcal{H}_{\bullet} (S)} (E, F) $. We recall from \cite[\S~4]{Morel}, the motivic model structure on $\Spt (S)$. We say $Z \in \Spt (S)$ is \emph{$\mathbb{A}^1$-local} if for any $E \in \Spt(S)$, the projection $E \wedge \mathbb{A}^1 _+ \to E$ induces an isomorphism of groups $[E,Z]_{\Nis} \simeq [E \wedge \mathbb{A}^1 _+, Z]_{\Nis}$. A morphism $f : E \to F$ in $\Spt(S)$ is an \emph{$S^1$-stable motivic weak-equivalence} if for each $\mathbb{A}^1$-local $Z$, the induced map $f^* : [F, Z]_{\Nis} \to [E, Z]_{\Nis}$ is an isomorphism. We often say that $f$ is a \emph{motivic weak-equivalence} of motivic spectra, for simplicity. 
The motivic weak-equivalences, cofibrations (as in \S \ref{subsubsection:NMSpec}) and \emph{motivic fibrations} (given by the right lifting property with respect to all trivial cofibrations), define a closed model structure on $\Spt (S)$, called the \emph{$S^1$-stable motivic model structure}. This model structure is the left localization of the Nisnevich local injective model structure with respect to the maps $E \wedge \mathbb{A}^1 _+ \to E$ for $E \in \Spt (S)$. By \cite[Proposition~3.4]{Hirschhorn} and Theorem \ref{thm:MSSpectra}, the motivic model structure on $\Spt (S)$ is proper and simplicial. 
A motivic spectrum is motivic fibrant if and only if it is Nisnevich fibrant and $\A^1$-local. Let $\mathcal{SH}_{S^1}(S)$ be the homotopy category of $\Spt(S)$ with respect to the $S^1$-stable motivic model structure. This model structure is equivalent to the one obtained by stabilizing the motivic model structure on $\Spc_{\bullet}(S)$ with respect to $\Sigma_s$, as described in \cite[Theorem~1.1]{Jardine}. It follows that $E \in \Spt(S)$ is motivic fibrant if and only if it is level-wise motivic fibrant in the motivic model structure on $\Spc_{\bullet}(S)$, and each map $E_n \to \Omega_sE_{n+1}$ is a motivic weak-equivalence. By \cite[Proposition~3.1.1]{Morel}, the category $\mathcal{SH}_{S^1} (S)$ is triangulated, where the shift functor $E \mapsto E[1]$ is $\Sigma_s$. We let $[E, F]_{\A^1}:=\Hom_{\mathcal{SH}_{S^1} (S)} (E,F)$. We have a pair of adjoint functors $\Sigma_s ^{\infty} : \Spc_{\bullet} (S) \leftrightarrow \Spt(S) : Ev^0_{s}$ given by $\Sigma_s ^{\infty}(E) = (E, \Sigma_s E, \Sigma_s ^2 E, \cdots)$ and $Ev^0_{s}(F) = F_0$. The functor $\Sigma_s ^{\infty}$ clearly preserves cofibrations. For $E \in \Spc_{\bullet}(S), F \in \Spt(S)$ and $p \in \Z$, there are natural isomorphisms (\emph{cf.} \cite[Theorem~5.2]{Voevodsky})
\begin{equation}\label{eqn:Lim-iso}
[\Sigma^{\infty}_sE[p], F]_{?} \simeq {\underset{n \ge -p}\colim} \ [S^{n+p}\wedge E, F_n]_{?}, \ \ ? = Nis \mbox{ or } \mathbb{A}^1,
\end{equation}
so that the functor $\Sigma_s ^{\infty}$ preserves motivic weak-equivalences. Thus, the pair $(\Sigma^{\infty}_s, Ev^0_{s})$ forms a Quillen pair, and one has adjoint functors $\Sigma_s ^{\infty} : \mathcal{H}_{\bullet} (S) \leftrightarrow \mathcal{SH}_{S^1} (S) : {\bf R}Ev^0_{s}$.

\subsection{Stable motivic homotopy category}\label{subsection:SHC}
The stable motivic homotopy category $\mathcal{SH} (S)$ was first constructed in \cite{Voevodsky}. It has several models. We review two such models. For $F \in \Spt(S)$ and $E \in \Spc_{\bullet}(S)$, we let $\Sigma_E F$ denote the motivic spectrum $(F_0\wedge E, F_1 \wedge E, \cdots )$. For $E = S^1$, the spectrum $\Sigma_E F$ is denoted by $\Sigma_sF$. Let $T= (\mathbb{P}^1, \infty)$.

\subsubsection{$(s,\mathfrak{p})$-bispectra model}\label{subsubsection: bispec}
Recall from \cite[\S8]{Levine coniveau} that an \emph{$(s,\mathfrak{p})$-bispectrum} over $S$ is a collection $E = \{ E_{m,n} \in \Spc_{\bullet} (S)  | m, n \geq 0 \}$ with the horizontal maps $\Sigma_s E_{m, n} \to E_{m+1, n}$ and the vertical maps $\Sigma_T E_{m, n} \to E_{m, n+1}$ such that the horizontal and the vertical maps commute. We regard it as a sequence $(E_0, E_1, \cdots)$, with the bonding maps $ \Sigma_TE_n \to E_{n+1}$, where $E_n \in \Spt(S)$ is $E_{*,n}:=(E_{0, n}, E_{1, n}, \cdots)$. Let $\Spt_{(s,\mathfrak{p})}(S)$ be the category of $(s,\mathfrak{p})$-bispectra over $S$. Given $E \in \Spt_{(s,\mathfrak{p})}(S)$ and $p,q \in \mathbb{Z}$, define $\pi_{p,q} (E)$ to be the presheaf $U \mapsto (\pi_{p,q} (E))(U) =\colim_n \Hom_{\mathcal{SH} _{S^1} (S)} (\Sigma_s ^{p-2q} \Sigma_T ^{q+n} \Sigma_s ^{\infty} U_+, E_n).$ A morphism $f: E \to F$ in $\Spt_{(s, \mathfrak{p})}(S)$ is called a \emph{stable motivic weak-equivalence} if the induced morphism $f_*: \pi_{p,q} (E) \to \pi_{p,q} (F)$ of presheaves is a stalk-wise isomorphism of groups on $(\Sm_S)_{\Nis}$. We often drop the word \emph{stable} for simplicity. 
There is a closed model structure on $\Spt_{(s,\mathfrak{p})} (S)$ ({\emph{cf.} \cite[\S~3]{Hovey}, \cite[\S~8.2]{Levine coniveau}), whose weak-equivalences are stable motivic weak-equivalences, called the \emph{stable motivic model structure}. By \cite[Proposition~1.14]{Hovey}, this model structure is obtained as a Bousfield localization of the level-wise model structure on $\Spt_{(s,\mathfrak{p})} (S)$ in which weak-equivalences (fibrations) are $T$-level-wise $S^1$-stable motivic weak-equivalences (motivic fibrations) in $\Spt(S)$, and $E \to F$ is a \emph{cofibration} if the maps $E_0 \to F_0$ and $E_{n+1} \coprod_{\Sigma_T E_n} \ \Sigma_T F_{n} \to F_{n+1}$
are cofibrations in the $S^1$-stable motivic model structure on $\Spt(S)$ for $n \ge 0$. This model structure on $\Spt_{(s,\mathfrak{p})} (S)$ is proper and simplicial. By \cite[Theorem~3.4]{Hovey}, we know $E \in \Spt_{(s,\mathfrak{p})}(S)$ is stable motivic fibrant if and only if each $E_n$ is $S^1$-stable motivic fibrant and the maps $ E_n \to \Omega_TE_{n+1}$ are $S^1$-stable motivic weak-equivalences for $n \ge 0$.

\subsubsection{$T$-spectra model}\label{subsubsection:T-spec-M}
A \emph{$T$-spectrum} $E$ over $S$ is a collection $(E_0, E_1, \cdots)$, $E_i \in \Spc_{\bullet} (S)$, with the maps  $\Sigma_T E_n \to E_{n+1}$. They form the category $\Spt_T (S)$. For $p,q \in \mathbb{Z}$, define the presheaf $\pi_{p,q} (E)$ on $\Sm_S$ by $U \mapsto (\pi_{p,q} (E))(U)=\colim_n \Hom_{\mathcal{H}_{\bullet} (S)} (\Sigma_s ^{p-2q} \Sigma_T ^{q+n} U_+, E_n).$ There is a proper simplicial closed model structure on $\Spt_T (S)$ in which $E \to F$ is a weak-equivalence if the induced map $f_*: \pi_{p,q} (E) \to \pi_{p,q} (F)$ is a stalk-wise isomorphism of groups on $(\Sm_S)_{\Nis}$. This model structure is obtained as a Bousfield localization of the model structure on $\Spt_{T} (S)$ in which weak-equivalences (fibrations) are level-wise motivic weak-equivalences (motivic fibrations) in $\Spc_{\bullet} (S)$. Given a motivic spectrum $E$, let $\Omega^{\infty}_sE : = \colim_m \ \Omega^m_s E_{m}$. For  a $T$-spectrum $E$, let $\Omega^{\infty}_TE : = \colim_m  \Omega^m_T E_{m}$. A $T$-spectrum $E = (E_0, E_1, \cdots )$ defines an $(s,\mathfrak{p})$-bispectrum $\sE : = (\Sigma^{\infty}_sE_0, \Sigma^{\infty}_s E_1, \cdots)$ by taking the level-wise simplicial infinite suspensions. Conversely, given an $(s,\mathfrak{p})$-bispectrum $\sF = (F_0, F_1, \cdots)$, we obtain a $T$-spectrum $F = \left(\Omega^{\infty}_s F_0 , \Omega^{\infty}_s F_1, \cdots \right)$. The correspondence $\Sigma^{\infty}_s :\Spt_T(S) \leftrightarrow \Spt_{(s,\mathfrak{p})}(S) : \Omega^{\infty}_s$ induces an equivalence between the homotopy categories of $\Spt_T (S)$ and $\Spt_{(s,\mathfrak{p})}(S)$. We write $\mathcal{SH} (S)$ for the common homotopy category. For $X \in \Spc_{\bullet} (S)$, one associates the infinite $T$-suspension spectrum, defined by $\Sigma_T ^{\infty} X:= (X,  \Sigma_T X , \Sigma^2_T X, \cdots)$, with the identity bonding maps $T \wedge T^{\wedge (n-1)} \wedge X \to T^n \wedge X$. We have suspension operations $\Sigma_T, \Sigma_s, \Sigma_{t}$ to $\Spt_{(s,\mathfrak{p})} (S)$ and $\Spt_T (S)$. The category $\mathcal{SH} (S)$ is triangulated with the shift functor $E \mapsto E[1]$ given by $\Sigma_s$, and all functors $\Sigma_T, \Sigma_s, \Sigma_{t}$ are auto-equivalences. For $E, F \in \Spt_{(s,\mathfrak{p})}(S)$, let $[E, F]_{\A^1}:= \Hom_{\mathcal{SH} (S)}(E, F)$. There is a Quillen pair $\Sigma^{\infty}_T : \Spt(S) \leftrightarrow \Spt_{(s,\mathfrak{p})}(S) : \Omega^{\infty}_T$ given by $\Sigma_s ^{\infty}(E) = (E, \Sigma_T E, \Sigma^2_T E, \cdots)$ and $\Omega^{\infty}_T(F) = (\Omega^{\infty}_T F_{0,*}, \Omega^{\infty}_T F_{1,*}, \cdots )$. This yields an adjoint pair of derived functors $\Sigma^{\infty}_T : \mathcal{SH}_{S^1} (S) \leftrightarrow \mathcal{SH} (S): {\bf R}\Omega^{\infty}_T.$ For $F \in \Spt_{(s,\mathfrak{p})}(S)$, one has ${\bf R}\Omega^{\infty}_T(F) =\Omega^{\infty}_T(\tilde{F}) = \tilde{F}_0$, where $F \to \tilde{F}$ is a stable motivic fibrant replacement of $F$, and $\tilde{F}_0 \in \Spt(S)$ is given by $ (\tilde{F}_{0,0}, \tilde{F}_{1,0}, \cdots).$

\subsubsection{Cohomology theories}\label{sec:cohomology of space}
Given $E, F \in \mathcal{SH} (S)$, the $E$-cohomology of $F$ is defined by $E^{a,b} (F):= [F, \Sigma ^{a,b} E]_{\A^1}$, where $a,b \in \mathbb{Z}$, $\Sigma ^{a,b} E: = \Sigma^{a-b}_s\Sigma^b_tE$. For $X \in \Sm_S$, using the object $\Sigma_T ^{\infty} X_+ \in \mathcal{SH} (S)$ we define 
$E^{a,b} (X):= E^{a,b} (\Sigma_T ^{\infty} X_+) = [\Sigma_T ^{\infty} X_+, \Sigma ^{a,b} E]_{\A^1} = [\Sigma^{\infty}_T \Sigma^{\infty}_s X_{+}, \Sigma^{a-2b}_s \Sigma^b_T E]_{\A^1}.$

\section{Motivic descent for $\A^1$-B.G. presheaves}\label{sec:A1BG}
Recall that a presheaf $E$ on $\Sm_S$ of objects in $\Spc$, $\Spc_{\bullet}$, or $\Spt$ has the \emph{B.G. property} if $E$ turns every Nisnevich square (\cite[Definition~3.1.5]{MV}) in $\Sm_S$
\begin{equation}\label{eqn:elementary square}
\begin{CD}
W @>>> U\\
@VVV @VV{p}V \\
V @>{j}>> X
\end{CD}
\end{equation} 
into a homotopy Cartesian square in $\Spc$, $\Spc_{\bullet}$, or $\Spt$. Recall the following (see \cite[Proposition~3.1.16]{MV}, \cite[Theorem~1.3, Corollary~1.4]{Jardine}), which gives a necessary and sufficient condition for a Nisnevich fibrant replacement to be an object-wise weak-equivalence.

\begin{thm}[Nisnevich descent theorem]\label{thm:Nis descent thm}
A presheaf $E$ on $\Sm_S$ of objects in $\Spc$, $\Spc_{\bullet}$, or $\Spt$ is B.G. if and only if every Nisnevich fibrant replacement $ E \to F$ is an object-wise weak-equivalence. A Nisnevich local weak-equivalence $E \to F$ of B.G. presheaves of objects in $\Spc, \Spc_{\bullet}$, or $\Spt$ is an object-wise weak-equivalence.
\end{thm}

\subsection{Motivic descent theorem}\label{subsection:MDT}
We establish a necessary and sufficient condition for a motivic fibrant replacement to be an object-wise weak-equivalence. 
Recall the following notion from \cite[Definition A.5]{Morel book}:

\begin{defn}\label{def:A1BG}Let $E$ be a presheaf on $\Sm_S$ of objects in $\Spc, \Spc_{\bullet}$, or $\Spt$. We say $E$ is \emph{$\mathbb{A}^1$-weak-invariant} if the map $E(X) \to E(X \times \mathbb{A}^1)$ induced by the projection is a weak-equivalence for all $X \in \Sm_S$. We say $E$ is \emph{$\mathbb{A}^1$-B.G.} if it is B.G. and $\mathbb{A}^1$-weak-invariant. We say $E$ is \emph{quasi-fibrant} (resp. \emph{motivic quasi-fibrant}) if every Nisnevich fibrant (resp. motivic fibrant) replacement $E \to F$ of $E$ is an object-wise weak-equivalence.
\end{defn}

Theorem \ref{thm:Nis descent thm} says $E$ is B.G. if and only if $E$ is quasi-fibrant. Let's begin with: 

\begin{lem}\label{lem:A1Nis-fib pi} Let $X \in \Sm_S$. $(1)$ If $F$ in $\Spc_{\bullet} (S)$ (resp. $\Spt(S)$) is Nisnevich fibrant, then we have a bijection $[S^p \wedge X_+, F]_{\Nis} \simeq \pi_p (F(X))$ (resp. $[\Sigma_s ^{\infty} X_+[p], F]_{\Nis} \simeq \pi_p (F(X))$).
$(2)$ If $F$ in $\Spc_{\bullet} (S)$ (resp. $\Spt(S)$) is motivic fibrant, then we have a bijection $[S^p \wedge X_+, F]_{\A^1} \simeq \pi_p (F(X))$ (resp. $[\Sigma_s ^{\infty} X_+[p], F]_{\A^1} \simeq \pi_p (F(X))$). 
\end{lem}
\begin{proof}
For $X \in \Sm_S$, the functors $ev_X: \Spc_{\bullet} (S) \leftrightarrow  \Spc_{\bullet} : sm_X$ given by $(ev_X: F \mapsto F(X))$ and $(sm_X : K \mapsto K \wedge X_{+})$ form a Quillen pair with respect to the Nisnevich local injective model structure and motivic model structure on $ \Spc_{\bullet} (S)$. In particular, their derived functors induce an adjoint pair of functors on the homotopy categories. The first isomorphism of $(1)$ follows immediately from this if $F \in \Spc_{\bullet} (S)$ is Nisnevich fibrant and the first isomorphism of $(2)$ follows if $F \in \Spc_{\bullet} (S)$ is motivic fibrant. The second isomorphisms of $(1)$ and $(2)$ follow from the first set of isomorphisms by applying Theorem~\ref{thm:MSSpectra} and ~\eqref{eqn:Lim-iso}.
\end{proof}

The following result follows immediately from Theorem \ref{thm:Nis descent thm} and \cite[Lemma~4.1.4]{Morel}.

\begin{lem}\label{lem:BG A1wi A1loc}Let $E$ be a B.G. presheaf on $\Sm_S$ of objects in $\Spc, \Spc_{\bullet}$, or $\Spt$. $(1)$ Let $E \to E'$ be a Nisnevich fibrant replacement. Then $E$ is $\mathbb{A}^1$-weak-invariant if and only if so is $E'$. $(2)$ $E$ is $\mathbb{A}^1$-weak-invariant if and only if $E$ is $\mathbb{A}^1$-local.
\end{lem}

\begin{lem}\label{lem:A1we Niswe}
A motivic fibrant replacement of an $\mathbb{A}^1$-B.G. presheaf on $\Sm_S$ of objects in $\Spc, \Spc_{\bullet}$, or $\Spt$ is also a Nisnevich fibrant replacement.
\end{lem}
\begin{proof}
We consider the case of presheaves of spectra as the other cases are similar. Let $f: E \to F$ be a motivic fibrant replacement. Since $F$ is Nisnevich fibrant and since cofibrations in the motivic model structure are Nisnevich cofibrations, it suffices to show that $f$ is a Nisnevich local weak-equivalence. Factor $f$ as a composition $f' \circ g: E  {\to}  E' {\to} F$, where $g$ is a Nisnevich trivial cofibration (thus, a motivic trivial cofibration) and $f'$ is a Nisnevich fibration. By the 2-out-of-3 axiom, $f'$ is a motivic weak-equivalence. We need to show that $f'$ is a Nisnevich local weak-equivalence. Since $F$ is Nisnevich fibrant and $f'$ is a Nisnevich fibration, it follows that $E'$ is Nisnevich fibrant. Thus, $g$ defines a Nisnevich fibrant replacement of $E$. Moreover, by Lemma~\ref{lem:BG A1wi A1loc}, we see that $E'$ is $\A^1$-local. Hence $E'$ is motivic fibrant. Now by Lemma~\ref{lem:A1Nis-fib pi}, $f': E' \to F$ is an object-wise weak-equivalence, thus a Nisnevich local weak-equivalence. 
\end{proof}

\begin{thm}[Motivic descent theorem]\label{thm:A1Nis descent}
Let $E$ be a presheaf on $\Sm_S$ of objects in $\Spc, \Spc_{\bullet}$, or $\Spt$. Then, $E$ is $\mathbb{A}^1$-B.G. if and only if it is motivic quasi-fibrant. A motivic weak-equivalence of $\A^1$-B.G. presheaves is an object-wise weak-equivalence.
\end{thm}

\begin{proof}
Suppose that $E$ is motivic quasi-fibrant. Let $f:E \to E'$ be a motivic fibrant replacement. Then $E'$ is Nisnevich fibrant (thus B.G.) and $\A^1$-local. So, by Lemma~\ref{lem:BG A1wi A1loc}, $E'$ is $\A^1$-B.G. Since $E$ is motivic quasi-fibrant, $f$ is an object-wise weak-equivalence, thus a Nisnevich local weak-equivalence. So, by Theorem \ref{thm:Nis descent thm}, $E$ is B.G., and by Lemma \ref{lem:BG A1wi A1loc}, it is $\mathbb{A}^1$-weak-invariant, that is, $E$ is $\mathbb{A}^1$-B.G. Conversely, suppose $E$ is an $\A^1$-B.G. Let $f: E \to E'$ be a motivic fibrant replacement. By Lemma~\ref{lem:A1we Niswe}, $f$ is also a Nisnevich fibrant replacement. That $f$ is an object-wise weak-equivalence follows now from Theorem~\ref{thm:Nis descent thm}. Thus, $E$ is motivic quasi-fibrant. This proves the first assertion. To prove the second one, given a motivic weak-equivalence $f:E \to F$ of $\mathbb{A}^1$-B.G. presheaves, form a commutative diagram
$$\xymatrix{ E \ar[r]   \ar[d] & F \ar[d] \\ E' \ar[r] ^{f'} & F',}$$
where the vertical arrows are motivic fibrant replacements, which are object-wise weak-equivalences by the first part. By the $2$-out-of-$3$ axiom, $f'$ is a motivic weak-equivalence. In this case, we have shown in the proof of Lemma~\ref{lem:A1we Niswe} that $f'$ is an object-wise weak-equivalence. But, we saw that two vertical arrows are also object-wise weak-equivalences. Thus, $f$ is an object-wise weak-equivalence. 
\end{proof}

\begin{cor}\label{cor:Descent-CR-1}
The isomorphisms in Lemma~\ref{lem:A1Nis-fib pi}(1) hold for all B.G. pointed motivic spaces and spectra, while Lemma~\ref{lem:A1Nis-fib pi}(2) hold for all $\A^1$-B.G. ones.
\end{cor}

\begin{cor}\label{cor:BG-colim}
The class of motivic quasi-fibrant presheaves on $\Sm_S$ of objects in $\Spc, \Spc_{\bullet}$, or $\Spt$ is closed under taking filtered colimits.
\end{cor}
\begin{proof}
This follows by Theorem~\ref{thm:A1Nis descent} and the proof of \cite[Corollary~4.2.7]{Morel}.
\end{proof}

For $E = \left(E_0, E_1, \cdots \right) \in \Spt(S)$, let $E\{n\}$ denote the motivic spectrum $\left(E_n, E_{n+1}, \cdots \right)$. Let $m \geq -1$. 
We say that $E$ is an \emph{object-wise (resp. motivic) $\Omega_s$-spectrum above level $m$} if the map 
$E_n \to \Omega_sE_{n+1}$ is an object-wise (resp. motivic) weak-equivalence for each $n > m$. 
An object-wise (resp. motivic) $\Omega_s$-spectrum above level  $m = -1$ will be called an object-wise (resp. motivic) $\Omega_s$-spectrum.

\begin{cor}\label{cor:A1BG consequence}
Suppose $E \in \Spt(S)$ is $\A^1$-B.G. Let $E \to F$ be a motivic fibrant replacement. $(1)$ For each $m, n, p \ge 0$, the map $\Omega^m_sF_{n} \to \Omega^{m+p} _s F_{n+p}$ is an object-wise weak-equivalence. $(2)$ For each $m, n \ge 0$, the motivic spectrum $\Omega^m_sF\{n\}$ is $S^1$-stable motivic fibrant. $(3)$ For each $n > m$, the map $E_n \to F_n$ is an object-wise weak-equivalence if $E$ is an object-wise $\Omega_s$-spectrum above level $m$.
\end{cor}

\begin{proof}
A motivic spectrum is $S^1$-stable motivic fibrant if and only if it is level-wise motivic fibrant and a motivic $\Omega_s$-spectrum. Thus, each $F_n \in \Spc_{\bullet} (S)$ is motivic fibrant. Since $S^1$ is cofibrant, we see that each $\Omega^m_sF_n$ is also motivic fibrant and the map $F_n  \to \Omega_s F_{n+1}$ is a motivic weak-equivalence. In particular, the map ${\bf R}\Omega^m_s F_n \to {\bf R}\Omega^{m+p}_sF_{n+p}$ is a motivic isomorphism. Since each $\Omega^m_sF_n$ is motivic fibrant, each map $\Omega^m_s F_n \to \Omega^{m+p}_sF_{n+p}$ is a motivic weak-equivalence for $m,n,p \ge 0$. By Lemma~\ref{lem:A1Nis-fib pi}, this map is an object-wise weak-equivalence, proving (1). Since each $\Omega^m_sF_n$ is motivic fibrant and the map $\Omega^m_sF_{n+p} \to \Omega^{m+1} _s F_{n+p+1}$ is a motivic weak-equivalence, it follows that $\Omega^m_sF\{n\}$ is $S^1$-stable motivic fibrant, proving $(2)$. For $(3)$, we first apply Theorem~\ref{thm:A1Nis descent} to deduce that $E \to F$ is an object-wise stable weak-equivalence. For $n > m, p \ge 0$ and $X \in \Sm_{S}$, we get isomorphisms $\pi_p\left(E_n(X)\right)  {\simeq}^1  \colim_{q }  \pi_{p+q}\left(E_{n+q}(X)\right)    \simeq  \pi_{p-n}\left(E(X)\right)  \simeq  \pi_{p-n}\left(F(X)\right)    {\simeq}^2   \pi_p\left(F_n(X)\right),$
where ${\simeq}^1$ holds because $E$ is an object-wise $\Omega_s$-spectrum above level $m$, and ${\simeq}^2$ holds because $F$ is an object-wise $\Omega_s$-spectrum.
\end{proof}

\subsection{$\mathbb{A}^1$-B.G. property of motivic spaces and motivic spectra}\label{sec:A1BG for spectra}
We study the $\A^1$-B.G. property of $E \in \Spt(\C)$ in terms of the property of its levels. 
Given any $E \in \Spc_{\bullet}(S)$,  $K \in \Spc_{\bullet}$ and $U \in \Sm_S$, there is an isomorphism ${\mathcal{H}om}_{\bullet} (K, E)(U) \simeq \Hom_{\bullet}(K, E(U))$ in $\Spc_{\bullet}$. Thus, we have $(\Omega_s E)(U) \simeq \Omega_s (E(U))$. Since $\Hom_{\bullet}(S^1, -)$ preserves weak-equivalences and fibration sequences in $\Spc_{\bullet}$, we deduce:

\begin{cor}\label{cor:A1BGOmega 1}
The functor $\Omega_s(-)$ preserves object-wise weak-equivalences, B.G. property, and $\A^1$-weak-invariance of $\Spc_{\bullet} (S)$. It preserves motivic weak-equivalences of $\A^1$-B.G. pointed motivic spaces. If $E \in \Spc_{\bullet}(\C)$ is $\A^1$-B.G., then natural map $\Omega_s E \to {\bf R}\Omega_s E$ is an isomorphism in $\sH_{\bullet}(S)$.
\end{cor}
\begin{proof}
The first statement is obvious. The second one follows from the first and Theorem~\ref{thm:A1Nis descent}. To see the last one, take a motivic fibrant replacement $E \to E'$, apply the second one, and use the isomorphism $\Omega_s E' \simeq {\bf R} \Omega_s E'$.
\end{proof}

We say $E\in \Spt(S)$ is \emph{level-wise $\mathbb{A}^1$-B.G.} if each $E_n$ is $\mathbb{A}^1$-B.G.

\begin{cor}\label{cor:A1BGOmega 4} Let $E \to F$ be a level-wise motivic weak-equivalence of level-wise $\mathbb{A}^1$-B.G. motivic spectra. If $E$ is a motivic $\Omega_s$-spectrum, then so is $F$. 
\end{cor} 
\begin{proof}
This is an immediate consequence of Theorem~\ref{thm:A1Nis descent} and Corollary~\ref{cor:A1BGOmega 1}.
\end{proof}

\begin{lem}\label{lem:A1BGOmega spt}
Let $f:E \to F$ be a morphism of level-wise $\mathbb{A}^1$-B.G. motivic $\Omega_s$-spectra on $\Sm_S$. Then $f$ is an $S^1$-stable motivic weak-equivalence if and only if each $f_n: E_n \to F_n$ is an object-wise weak-equivalence.
\end{lem}

\begin{proof}
Suppose that $f:E \to F$ is an $S^1$-stable motivic weak-equivalence. Let $n, p  \ge 0$ and $U \in \Sm_S$. Since $E$ and $F$ are level-wise $\mathbb{A}^1$-B.G., by Corollary \ref{cor:Descent-CR-1}, $\pi_p(E_n(U)) \simeq  [S^p \wedge U_{+}, E_n]_{\A^1}    {\simeq}^1   [S^p \wedge U_{+}, \Omega_s^{m-n}E_{m}]_{\A^1}
 {\simeq}^{2}   [S^p \wedge U_{+}, {\bf R}\Omega_s^{m-n}E_{m}]_{\A^1}  {\simeq}^{3}   [S^{m+p-n}  \wedge U_{+}, E_m]_{\A^1}$, where ${\simeq}^{1}$ holds because $E$ is a motivic $\Omega_s$-spectrum, ${\simeq}^{2}$ holds by Corollary~\ref{cor:A1BGOmega 1}, ${\simeq}^3$ holds by the adjointness. But $m \gg 0$ is arbitrary so $[S^{m+p-n}  \wedge U_{+}, E_m]_{\A^1} = \colim_{m}[S^{m+p-n}  \wedge U_{+}, E_m]_{\A^1}$, which is $ [\Sigma^{\infty}_sU_{+}[p-n], E]_{\A^1}$ by \eqref{eqn:Lim-iso}. Similarly, $\pi_p(F_n(U)) \simeq [\Sigma^{\infty}_sU_{+}[p-n], F]_{\A^1}$. Since $f$ is an $S^1$-stable motivic weak-equivalence, we deduce that the map $f_n: E_n \to F_n$ is an object-wise weak-equivalence. The other direction is obvious.
\end{proof}

\begin{cor}\label{cor:motivic descent S^1 version}Every level-wise $\A^1$-B.G. motivic $\Omega_s$-spectrum is motivic quasi-fibrant.
\end{cor}
\begin{proof}Consider an $S^1$-stable motivic fibrant replacement of the given one. Since an $S^1$-stable motivic fibrant motivic spectrum is a level-wise motivic fibrant (thus $\mathbb{A}^1$-B.G.) motivic $\Omega_s$-spectrum, this corollary holds by Lemma~\ref{lem:A1BGOmega spt} and Theorem \ref{thm:A1Nis descent}.
\end{proof}

\subsection{Motivic descent for $(s,\mathfrak{p})$-bispectra}\label{subsection:A1BG bispec}
Given an open or a closed immersion of schemes $A \subseteq B$ in $\Sm_{S}$, let $\Omega_{B/A}(-)$ be the functor $E \mapsto \Omega_{B/A}E = (\Omega_{B/A}E_0, \Omega_{B/A} E_1, \cdots )$ on $\Spt(S)$, where $\Omega_{B/A}F = {\sH}om_{\bullet}(B/A, F)$ is the object-wise fiber of the map ${\sH}om(B, F) \to {\sH}om(A, F)$ (\emph{cf.} \cite[Corollary~1.10]{Jardine}) for $F \in \Spc_{\bullet} (S)$. There is an object-wise fiber sequence of presheaves $\Omega_{B/A}E \to E_B \to E_A,$
where $E_B(X) : = E(B \times X) = {\sH}om(B, E)(X)$. Recall (\emph{cf.} \cite[Corollary~3.2]{Jardine}) that given an object-wise fiber sequence as above, the map ${E_B}/{(\Omega_{B/A}E)} \to E_A$ is an object-wise $S^1$-stable weak-equivalence. The natural isomorphism $S^1 \wedge E_X \to (S^1\wedge E)_X$, for $E \in \Spc_{\bullet}(S)$ and $X \in \Sm_{S}$, and the above cofiber sequence give a natural map $S^1 \wedge \Omega_{B/A} E_n \to \Omega_{B/A} (S^1 \wedge E_n)$ for $E \in \Spt(S)$. Composed with the bonding map $\Omega_{B/A} (S^1 \wedge E_n) \to \Omega_{B/A} (E_{n+1})$, we see that $E \to \Omega_{B/A} E$ is an endo-functor on $\Spt(S)$. There is a natural bijection $\Hom_{\Spt(S)}(\Sigma_{B/A} E, F) \simeq \Hom_{\Spt(S)}(E, \Omega_{B/A} F)$. The following analogue of Corollary \ref{cor:A1BGOmega 1} for motivic spectra is immediate from Theorem~\ref{thm:A1Nis descent} and the above object-wise cofiber sequence.

\begin{lem}\label{lem:A1BGOmega-Spec 1}
The functor $\Omega_{B/A} (-)$ preserves object-wise weak-equivalences, B.G. property, and $\A^1$-weak-invariance of motivic spectra. It preserves motivic weak-equivalences of $\A^1$-B.G. motivic spectra. If $E \in \Spt (S)$ is $\A^1$-B.G., then the natural map $\Omega_{B/A} E \to {\bf R}\Omega_{B/A} E$ is an isomorphism in $\sS\sH_{S^1}(S)$. If $f: E \to F$ is an $S^1$-stable motivic weak-equivalence of $\A^1$-B.G. motivic spectra, then $\Omega_{B/A} f: \Omega_{B/A} E \to \Omega_{B/A} F$ is also an $S^1$-stable motivic weak-equivalence.
\end{lem}

Recall from \S\ref{subsubsection: bispec}, \S\ref{subsubsection:T-spec-M} that an $(s,\mathfrak{p})$-bispectrum $E = (E_{m,n})_{m,n \ge 0}$ gives a sequence $ (E_0, E_1, \cdots )$ of motivic spectra with bonding maps $\Sigma_TE_n = T \wedge E_n \to E_{n+1}$.

\begin{defn}\label{defn:BG bispec}
For $E \in \Spt_{(s, \mathfrak{p})}(S)$, we say that $E$ is a \emph{motivic $\Omega_T$-bispectrum} if the adjoint maps $ E_n \to \Omega_T E_{n+1}$ are motivic weak-equivalences in $\Spt(S)$ for $n \ge 0$. We say that $E$ is \emph{$\mathbb{A}^1$-B.G.} if each $E_n$ is an $\mathbb{A}^1$-B.G. motivic spectrum for $n \geq 0$.
\end{defn}

\begin{thm}\label{thm:descent for bispectrum}
Let $f: E \to F$ be a stable motivic weak-equivalence of $\mathbb{A}^1$-B.G. motivic $\Omega_T$-bispectra on $\Sm_S$. Then $f$ is a $T$-level-wise object-wise weak-equivalence, i.e., each $f_n : E_n \to F_n$ is an object-wise weak-equivalence.
\end{thm}
\begin{proof}
Let $n \ge 0$, $p \in \Z$ and $U \in \Sm_S$. Since $E$ is ($T$-level-wise) $\A^1$-B.G., apply Corollary~\ref{cor:Descent-CR-1} to get $\pi_p(E_n)(U) \simeq  [\Sigma^{\infty}_s U_{+}[p], E_n]_{\A^1}   {\simeq}^1  [\Sigma^{\infty}_s U_{+}[p], \Omega^{m-n}_TE_{m}]_{\A^1},$ where $\simeq^1$ holds for $E$ is a motivic $\Omega_T$-bispectrum. This equals $[\Sigma^{\infty}_s  U_{+}[p], {\bf R}\Omega^{m-n}_TE_{m}]_{\A^1}$ by Lemma~\ref{lem:A1BGOmega-Spec 1}. By adjointness it equals $[\Sigma_T^{m-n} \Sigma_s^{p} \Sigma^{\infty}_s U_{+}, E_m]_{\A^1}$. Since $m \gg 0$ is arbitrary, we have $[\Sigma_T^{m-n} \Sigma_s^{p} \Sigma^{\infty}_s U_{+}, E_m]_{\A^1} = \colim_{m}\ [\Sigma_T^{m-n} \Sigma_s^{p} \Sigma^{\infty}_s U_{+}, E_m]_{\A^1}$, which is $\pi_{p-n, -n}(E)(U)$ by definition in \S \ref{subsection:SHC}. Similarly, $\pi_p(F_n(U)) \simeq \pi_{p-n, -n}(F)(U)$.  
Now, by our assumptions, the map $\pi_p(E_n) \to \pi_p(F_n)$ induces an isomorphism of the associated Nisnevich sheaves so that $f_n: E_n \to F_n$ is a Nisnevich local weak-equivalence, and hence an $S^1$-stable motivic weak-equivalence. Since these are $\A^1$-B.G. motivic spectra, by Theorem~\ref{thm:A1Nis descent} each $f_n$ is an object-wise weak-equivalence.
\end{proof}

\begin{cor}For $E \in \Spt_{(s,\mathfrak{p})} (S)$, let $f: E \to E'$ be a stable motivic fibrant replacement. Then, $E$ is an $\mathbb{A}^1$-B.G. motivic $\Omega_T$-bispectrum if and only if $f$ is a $T$-level-wise object-wise weak-equivalence.
\end{cor}

\begin{proof}
The forward direction is obvious by Theorem \ref{thm:descent for bispectrum}. For the backward direction, note that each level $E_n \to E'_n$ is an object-wise weak-equivalence, with $E'_n$ is motivic fibrant, so that each $E_n$ is $\mathbb{A}^1$-B.G. by Theorem \ref{thm:A1Nis descent}.
It only remains to see that $E$ is a motivic $\Omega_T$-bispectrum. This follows from Lemma \ref{lem:A1BGOmega-Spec 1}.
\end{proof}

\section{Singular semi-topologization}\label{sec:sst}

\subsection{Definition and basic properties}\label{sec:recall sst}
From now, we take $S= \Spec (\mathbb{C})$. For a complex algebraic variety $U$, let $U^{an}$ be its associated complex analytic space. 
We recall the semi-topologization of Friedlander-Walker from \cite[Definition 10]{FW Handbook}.

\subsubsection{Realization and diagonal of a simplicial spectrum}\label{subsubsection:Total}
We briefly review the diagonal and the realization of a simplicial spectrum. For a bisimplicial set $A_{* *}$, the \emph{realization} $|A|$ is the simplicial set obtained by taking the coequalizer of the diagram
$\coprod_{(\alpha: [n] \to [k]) \in \Delta^{\rm op}} \ A_n \times \Delta[k]  \rightrightarrows \coprod_{n \ge 0} \  A_n \times \Delta[n],$
where the two morphisms are $(\alpha, x, t) \mapsto (x, \alpha^*(t))$ and $(\alpha, x, t) \mapsto (\alpha_*(x), t)$. If $A_{**}$ is a simplicial object in $\Spc_{\bullet}$, then $|A|$ is obtained by replacing $A_n \times \Delta[k]$ by $A_n \wedge (\Delta[k])_{+}$ in the above. The \emph{diagonal $\diag A$} is the composite $ A_{**} \circ \delta: \Delta^{\rm op}\to  \Delta^{\rm op} \times \Delta^{\rm op} \to {\rm {\bf Set}}$. There is a natural isomorphism $\diag A \to |A|$. See \cite[Proposition~B.1]{BF}. If $E: \Delta^{\rm op} \to \Spt$ is a simplicial spectrum, its realization $|E|$ is defined as above, where $A_n \times \Delta[k]$ is replaced by $E(\Delta[n]) \wedge (\Delta[k])_{+}$. A simplicial spectrum $E$ can be seen as a sequence $(E^0_{**}, E^1_{**}, \cdots )$, where each $E^n_{**}$ is a pointed bisimplicial set, with the bonding maps $S^1 \wedge E^n_{**} \to E^{n+1}_{**}$. So, the spectrum $|E|$ has $|E|_n =|E^n_{**}|$ in $\Spc_{\bullet}$, with the bonding maps $S^1 \wedge |E^n_{**}| \to |E^{n+1}_{**}|$. The diagonal $\diag E$ of $E$ is the spectrum with $(\diag E)_n = \diag (E^n_{**})$. We have the map $S^1 \wedge E^n_{p} \to E^{n+1}_{p}$, where $E^n_{p} = \left(E(\Delta[p])\right)_n$, or, the map of pointed sets $(S^1)_i \wedge E^n_{p, i} \to E^{n+1}_{p,i}$. The maps $(S^1)_p \wedge E^n_{p, p} \to E^{n+1}_{p, p}$ give the bonding maps $S^1 \wedge (\diag E)_n \to (\diag E)_{n+1}$ of the spectrum $\diag E$. From the case of bisimplicial sets, one gets $\diag E \simeq |E|$.  If $E$ is a presheaf of simplicial spectra on $\Sch_S$ or $\Sm_S$, we define $|E|$ and $\diag E$ object-wise. Thus, for a simplicial presheaf of spectra $E$ on $\Sch_S$ or $\Sm_S$, we have $\diag E  \simeq |E|$.

\subsubsection{Semi-topologization}\label{subsection:SST-Def} 
For $T\in \mathcal{T}op$, let $(T|\Var_{\mathbb{C}})$ be the category whose objects are $(f, U)$, where $U \in \Var_{\mathbb{C}}$, and $f: T \to U^{an}$ is a continuous map. A morphism from $(f,U)$ to $(g, V)$ is a morphism $h: U \to V$ in $\Var_{\mathbb{C}}$ such that the map $h^{an}: U^{an} \to V^{an}$ satisfies $h^{an} \circ f = g$. Recall that $\Delta_{\top} ^{\bullet} = \{ \Delta_{\top} ^n \} _{n \geq 0}$ is a cosimplicial topological space with the natural cofaces $\partial ^i$ and the codegeneracies $s^i$. For $n>0$ and $0 \leq i \leq n$, define $\partial_i : (\Delta_{\top} ^n |\Var_{\mathbb{C}})^{\op} \to (\Delta_{\top} ^{n-1} |\Var_{\mathbb{C}})^{\op}$ by $(f: \Delta_{\top} ^n {\to} U^{an}) \mapsto ( f \circ \partial ^i : \Delta_{\top} ^{n-1} {\to} \Delta_{\top} ^n {\to} U^{an}).$
For $n \geq 0$ and $0 \leq i \leq n$, define $s_i : (\Delta_{\top} ^n |\Var_{\mathbb{C}})^{\op} \to 
(\Delta_{\top} ^{n+1}| \Var_{\mathbb{C}})^{\op}$ by $(f: \Delta_{\top} ^n {\to} U^{an}) \mapsto 
(f\circ s^i : \Delta_{\top} ^{n+1} {\to} \Delta_{\top} ^n {\to} 
U^{an}).$ 
Recall the following from \cite{FW Handbook}:

\begin{defn}\label{def:presst}
Let $E$ be a presheaf on $\Sch_{\mathbb{C}}$ of objects in $\Spc, \Spc_{\bullet}$, or $\Spt$. Let $X \in \Sch_{\mathbb{C}}$ and let $T\in \mathcal{T}op$. Define $E(T \times X)= \colim_{(f,U) \in (T|\Var_{\mathbb{C}})^{\op}} E(U \times X)$.  Consider $E(\Delta_{\top} ^{\bullet} \times X)=\{E (\Delta_{\top} ^n \times X)\}_{n \geq 0}$, which is a simplicial object in $\Spc, \Spc_{\bullet},$ or $\Spt$. Let $E^{\sst} (X):=|E(\Delta_{\top}^{\bullet} \times X)|$. This $E^{\sst}$ is a presheaf on $\Sch_{\mathbb{C}}$ of objects in $\Spc, \Spc_{\bullet}$, or $\Spt$, called \emph{the semi-topologization of $E$}.
\end{defn}

There is a natural morphism of presheaves $ E \to E^{\sst}$ on $\Sch_{\mathbb{C}}$, which gives a natural transformation ${\rm Id} \to (-)^{\sst}$ of functors on presheaves on $\Sch_{\mathbb{C}}$. 

\begin{lem}\label{lem:cor sub}
Let $E$ be a presheaf on $\Sch_{\mathbb{C}}$ of objects in $\Spc, \Spc_{\bullet}$, or $\Spt$. Let $X \in \Sch_{\mathbb{C}}$. Define a presheaf on $\Sch_{\mathbb{C}}$ by $E_X (U):= E( U \times X)$. Then, $(E_X)^{\sst} = (E^{\sst} )_X$. 
\end{lem}
\begin{proof}
For $U \in \Sch_{\C}$, we have $(E^{\sst})_X(U) = E^{\sst}(X \times U) = |{\{E(\Delta^{n}_{\rm top} \times X \times U)\}}_{n }| = |{\{\colim_{(f,C)} E(C\times X \times U)\}}_{n }|=|{\{\colim_{(f,C)} E_X(C\times U)\}}_{n }| = |{\{E_X(\Delta^{n}_{\rm top} \times U)\}}_{n }|$. This is by definition, $(E_X)^{\sst}(U)$.
\end{proof}

When $E$ is a presheaf on $\Sm_{{\C}}$, it is well-known that the realization of $E(\Delta^{\bullet} \times -)$ is $\mathbb{A}^1$-weak-invariant (\cite[Proposition 7.2]{FS}, \cite[p. 150]{FV}). Its semi-topological analogue also holds by \cite[Lemma 1.2]{FW0}:

\begin{thm}\label{thm:A1-invariance} Let $E$ be a presheaf on $\Sch_{\mathbb{C}}$ of objects in $\Spc, \Spc_{\bullet}$, or $\Spt$. Then, $E^{\sst}$ is $\mathbb{A}^1$-weak-invariant.
\end{thm}

\subsection{Semi-topologization and $\wedge$-product}
\label{subsection:SST-Wedge}

 Recall that for $A, B\in \Spc_{\bullet}$ (all base points are denoted by $\star$), we have $A \wedge B=(A \times B) / (A \vee B),$
where $A \vee B= (A \times \star) \cup (\star \times B)$ in $A \times B$. For two presheaves $E$ and $F$ on a category $\mathcal{C}$ of objects in $\Spc_{\bullet}$, define the presheaf $E \wedge F$ on $\mathcal{C}$ object-wise by $(E \wedge F) (U) = E (U) \wedge F(U)$, so one still has $E \wedge F = (E \times F) / (E \vee F)$. When $F$ is a presheaf of spectra on $\Sch_{\mathbb{C}}$ while $E$ is as above, we define $E \wedge F$ level-wise, namely, $E\wedge F= (E \wedge F_0, E \wedge F_1, \cdots)$.

\begin{prop}\label{prop:sst and wedge}
Let $E, F, F'$ be presheaves on $\Sch_{\mathbb{C}}$ of objects in $\Spc_{\bullet}$. Then we have the following identities: $(1)$ $(E \times F)^{\sst} = E^{\sst} \times F^{\sst}.$ $(2)$ $(E \vee F)^{\sst} = E^{\sst} \vee F^{\sst}.$ $(3)$ If $F \subset F'$, then $({F'}/F)^{\sst} = {F'}^{\sst} / F^{\sst}.$ $(4)$ $(E \wedge F)^{\sst} = E^{\sst} \wedge F^{\sst}$. $(5)$ When $E$ is as above and $F$ is a presheaf of spectra on $\Sch_{\mathbb{C}}$, we have $(E \wedge F)^{\sst} = E^{\sst} \wedge F^{\sst}$.
\end{prop}

\begin{proof} 
Let $X \in \Sch_{\mathbb{C}}$ be a fixed scheme. For (1), let $U \in \Sch_{\mathbb{C}}$. Note that $(E \times F ) (U \times X) = E(U \times X) \times F(U \times X)$.
Over the objects $(f : \Delta^n _{\top} \to U^{an})$ of the filtered category $ (\Delta^n _{\top} | \Var_{\mathbb{C}})^{\op}$, take the filtered colimit. By \cite[\S~IX.2 Theorem 1]{ML} finite limits (e.g. products) commute with filtered colimits, so that $ (E \times F) (\Delta_{\top} ^n \times X) = E (\Delta_{\top} ^n \times X) \times F( \Delta_{\top} ^n \times X).$ Taking the diagonals, we obtain (1).
For (2), for each $U \in \Sch_{\mathbb{C}}$, note that $(E\vee F)(U \times X) = \colim \{ E (U \times X) \times \star \leftarrow \star \times \star \rightarrow \star \times F(U \times X ) \}$. Take the filtered colimits over the objects $(f : \Delta^n _{\top} \to U^{an})$ of $ (\Delta^n _{\top} | \Var_{\mathbb{C}})^{\op}$. Colimits commute among themselves (see \cite[\S~IX.8]{ML}) so that $(E \vee F) (\Delta_{\top} ^n \times X) = E (\Delta_{\top} ^n \times X) \vee F (\Delta_{\top} ^n \times X)$. This implies (2), by taking the diagonals. For (3), similarly we consider instead $F' (U \times X) / F(U \times X) = \colim \{ \star \leftarrow F(U \times X) \rightarrow F' (U \times X) \}$, and repeat the same procedure. This proves (3). 
Now, (4) follows from (1) - (3). For (5), since the limits and colimits of spectra are all defined level-wise, this part follows from (4).
\end{proof}

\section{Semi-topologization of presheaves on smooth schemes}\label{sec:sst smooth}

\subsection{Artificial extension}\label{subsec:sst smooth}We discuss how one defines semi-topologization on presheaves on $\Sm_{\mathbb{C}}$. For a presheaf $E$ on $\Sch_{\mathbb{C}}$ of objects in $\Spc, \Spc_{\bullet}$, or $\Spt$, we used the categories $(\Delta_{\top} ^n | \Var_\mathbb{C})^{\op}$ to define $E^{\sst}$ in \S\ref{sec:sst}. If $E$ is defined only on $\Sm_\mathbb{C}$ \emph{a priori}, then one may either extend the functor $F$ to all of $\Sch_{\mathbb{C}}$, or shrink the indexing categories to, say, $(\Delta_{\top} ^n | \Sm_\mathbb{C})^{\op}$. Both raise some issues. Extension of $F$ from $\Sm_\mathbb{C}$ to $\Sch_{\mathbb{C}}$ is not unique. On the other hand, the inclusion $(\Delta_{\top} ^n | \Sm_{\mathbb{C}})^{\op}) \hookrightarrow (\Delta_{\top} ^n | \Var_{\mathbb{C}})^{\op}$ is not cofinal. Furthermore, the indexing categories $(\Delta_{\top}^n | \Sm_\mathbb{C})^{\op}$ are not filtered, over which the colimits have poor properties. To avoid these, we use a fixed functorial extension process to obtain a presheaf on $\Sch_{\mathbb{C}}$, and then apply the $\sst$-functor of \S\ref{sec:sst}.

Let $W \in \Sch_{\C}$. Consider the objects $(f, X)$, where $X \in \Sm_{\C}$ and $f: W \to X$ is a morphism in $\Sch_{\mathbb{C}}$. Given $(f, X)$ and $(g, Y)$, with $X, Y \in \Sm_{\C}$, a morphism $\psi$ from $(f,X)$ to $(g,Y)$ is defined to be a morphism $\psi : X \to Y$ in $\Sch_{\mathbb{C}}$ such that $\psi \circ f = g$. Let $(W|\Sm_\mathbb{C})^{\op}$ be the category of the pairs $(f, X)$ with the above morphisms. 

\begin{defn}\label{def:extension}Let $E$ be a presheaf on $\Sm_{\C}$ of objects in a cocomplete category $\mathcal{M}$. For $W \in \Sch_{\C}$, define \emph{the artificial extension} $\bar{E}$ of $E$ by $\bar{E} (W) := \colim_{(f,X) \in (W|\Sm_{\C}) ^{\op}} E(X).$
\end{defn}

If $W \in \Sm_{\C}$, then $(W|\Sm_{\C})^{\op}$ has the terminal object $({\rm Id}_W ,W)$ so that $\bar{E} (W)= E(W)$. One checks that given $\phi: W \to W'$ in $\Sch_{\C}$, the assignment $(f: W' \to X) \mapsto (f \circ \phi : W \to X)$ makes $\bar{E}$ a presheaf on $\Sch_{\C}$. One checks it defines a functor 
${\rm\bf ext}: {\rm Funct}(\Sm_{\C} ^{\op}, \mathcal{M}) \to  {\rm Funct} (\Sch_{\C} ^{\op}, \mathcal{M}).$
In the opposite direction, we have ${\rm\bf rest}: {\rm Funct} (\Sch_{\C} ^{\op}, \mathcal{M}) \to {\rm Funct }(\Sm_{\C} ^{\op}, \mathcal{M})$ and clearly ${\rm\bf rest} \circ {\rm\bf ext} = {\rm Id}$. The transformation ${\rm\bf ext} \circ {\rm\bf rest} \to {\rm Id}$ is not an isomorphism in general.

\begin{defn}\label{def:smsst}Let $E$ be a presheaf on $\Sm_{\mathbb{C}}$ of objects in $\Spc, \Spc_{\bullet}$, or $\Spt$. We define its \emph{semi-topologization} as the presheaf $\left({\rm\bf ext}(E)\right)^{\sst}|_{\Sm_{\C}} = {\bar{E}}^{\sst}|_{\Sm_{\C}} = {\rm\bf rest} \circ {\sst} \circ {\rm\bf ext}(E)$ on $\Sm_{\C}$. The resulting presheaf is denoted by $E^{\sst}$. 
\end{defn}
The semi-topologization defines a natural transformation of functors ${\rm Id} \to (-)^{\sst}$ on presheaves on $\Sm_{\C}$. Immediately from Theorem~\ref{thm:A1-invariance}, we get the following:

\begin{prop}\label{prop:spt basic sst}
Let $E \in \Spt(\mathbb{C})$. Then, $E^{\sst}$ is $\mathbb{A}^1$-weak-invariant.
\end{prop}

Recall the following important tool from \cite[Theorem~11]{FW Handbook}}, which is used in the form of Theorem \ref{thm:rec prin}.

\begin{thm}[Friedlander-Walker recognition principle]\label{thm:FWRP}
Let $E$ and $F$ be presheaves of spectra on $\Sch_{\mathbb{C}}$ and let $f: E \to F$ be a morphism of presheaves, which is an object-wise weak-equivalence on $\Sm_{\mathbb{C}}$. Then $|E(\Delta_{\top} ^{\bullet})| \to |F(\Delta_{\rm top} ^{\bullet})|$ is a weak-equivalence.
\end{thm}

\begin{thm}\label{thm:rec prin}
$(1)$ If $f:E \to F$ is a morphism of presheaves of spectra on $\Sch_{\C}$, which is an object-wise weak-equivalence on $\Sm_{\C}$, then $f^{\sst} : E^{\sst} \to F^{\sst}$ is an object-wise weak-equivalence on $\Sm_{\mathbb{C}}$. $(2)$ If $f:E \to F$ is a morphism of presheaves of spectra on $\Sm_{\C}$, which is an object-wise weak-equivalence, then $f^{\sst} : E^{\sst} \to F^{\sst}$ is an object-wise weak-equivalence on $\Sm_{\mathbb{C}}$.
\end{thm}
\begin{proof}
We first prove $(1)$. By the given assumption, for $X \in \Sm_{\C}$, the map $E_X \to F_X$ is an object-wise weak-equivalence on $\Sm_{\C}$. By Theorem~\ref{thm:FWRP}, the map $\left(E_X\right)^{\sst}(\Spec(\C)) \to \left(F_X\right)^{\sst}(\Spec(\C))$ is a weak-equivalence. Now by Lemma~\ref{lem:cor sub}, the map ${E}^{\sst}(X) \to {F}^{\sst}(X)$ is a weak-equivalence. To prove $(2)$, note that the map $\bar{E} \to \bar{F}$ is an object-wise weak-equivalence on $\Sm_{\C}$. So, by $(1)$ the map ${\bar{E}}^{\sst}(X) \to {\bar{F}}^{\sst}(X)$ is a weak-equivalence for $X \in \Sm_{\C}$. Equivalently, the map $E^{\sst}(X) \to F^{\sst}(X)$ is a weak-equivalence.
\end{proof}

\subsection{The loop space and the $\sst$-functor}\label{subsection:Loop-sst}
For a map $f:E \to F$ of presheaves on $\Sm_{\C}$ or $\Sch_{\C}$ of objects in $\Spc_{\bullet}$, the fiber ${\rm fib}(f)$ is by definition $\lim  \{ \star \rightarrow F {\leftarrow} E\}$, and ${\rm fib}(f) \to E \to F$ is called a fiber sequence. This is not same as a homotopy fiber sequence unless $f$ is a fibration.  Given an open or a closed immersion of schemes $A \subseteq B$ in $\Sm_{\C}$, the functor $\Omega_{B/A}(-)$ on $\Spt(\mathbb{C})$ is $E \mapsto \Omega_{B/A}E = (\Omega_{B/A}E_0, \Omega_{B/A} E_1, \cdots )$ (\S\ref{subsection:A1BG bispec}) , where $\Omega_{B/A}E_n ={\sH}om_{\bullet}(B/A, E_n) = {\rm fib}\left({\sH}om(B, E_n) \to {\sH}om(A, E_n)\right)$. So, we have an object-wise fiber sequence $\Omega_{B/A}E \to E_B \to E_A$ of presheaves, where $E_B(X) = E(B \times X) = {\sH}om(B, E)(X)$. For $B \in \Sm_{\C}$, the map ${\sH}om(B, E) \to {\sH}om(B, E^{\sst})$ induces ${\sH}om(B, E)^{\sst} \to {\sH}om(B, E^{\sst})$. By the universal property of $\Omega_{B/A}$, there is a natural transformation $(\Omega_{B/A}(-))^{\sst} \to \Omega_{B/A}((-)^{\sst})$. 

\begin{prop}\label{prop:sst of Gm loop}
For $E \in \Spt(\C)$, the map $(\Omega_{B/A}E)^{\sst} \to \Omega_{B/A}(E^{\sst})$ in $\Spt(\C)$ is an object-wise weak-equivalence on $\Sm_{\C}$. 
\end{prop}
\begin{proof}For $E \in \Spt(\C)$, let $\bar{E}={\rm \bf ext}(E)$, and $\bar{E}_B (X):= \bar{E}(B \times X)$ for $X \in \Sch_{\mathbb{C}}$.  
Let ${\Omega}_{B/A}\bar{E}$ be the object-wise fiber of $\bar{E}_B \to \bar{E}_A$. 
This sequence on $\Sch_{\C}$ restricts to $\Omega_{B/A}E \to E_B \to E_A$ on $\Sm_\C$. Note that there is a morphism $\ov{E_B} \to \bar{E}_B$ of presheaves on $\Sch_{\C}$, which restricts to ${\rm Id} : E_B \to E_B$ on $\Sm_{\C}$. By the universal property of fiber, we get a morphism of presheaves ${\Omega_{B/A}' E}:= {\rm fib}\left(\ov{E_B} \to \ov{E_A}\right) \to \Omega_{B/A}\bar{E}$ on $\Sch_{\C}$, which is an isomorphism on $\Sm_{\C}$, which gives a commutative diagram of presheaves on $\Sch_{\mathbb{C}}$
\begin{equation}\label{eqn:Hom-spec-1}
\xymatrix@C2pc{
\ov{\Omega_{B/A} E} \ar[r] \ar[d] \ar@/_2pc/[dd]_u & \ov{E_B} \ar[r] \ar@{=}[d] &  
\ov{E_A} \ar@{=}[d] \\
{\Omega_{B/A}' E} \ar[r] \ar[d] & \ov{E_B} \ar[r] \ar[d] &  \ov{E_A} \ar[d] \\
\Omega_{B/A}\bar{E} \ar[r] & \bar{E}_B \ar[r] & 
\bar{E}_A,}
\end{equation}
where the bottom two rows are object-wise fiber sequences and all vertical arrows are isomorphisms on $\Sm_{\C}$. Let $u$ be the composition. Since filtered colimits commute with fiber products (\cite[\S~IX.2 Theorem 1]{ML}), from the above we deduce the diagram
$$\xymatrix{
{\Omega_{B/A}' E}(\Delta^n_{\rm top} \times -) \ar[r] \ar[d] & 
\ov{E_B}(\Delta^n_{\rm top} \times -) \ar[r] \ar[d] &  
\ov{E_A}(\Delta^n_{\rm top} \times -) \ar[d] \\
\Omega_{B/A}\bar{E}(\Delta^n_{\rm top} \times -) \ar[r] & 
\bar{E}_B(\Delta^n_{\rm top} \times -) \ar[r] & 
\bar{E}_A(\Delta^n_{\rm top} \times -),}
$$
of presheaves of spectra on $\Sch_{\mathbb{C}}$, where the rows are object-wise fiber sequences. 
Since the fiber of a map of spectra is defined level-wise, taking the diagonals of maps of simplicial spectra as in \S \ref{subsubsection:Total}, we get a commutative diagram 
\begin{equation}\label{eqn:Hom-spec-3}
\xymatrix@C1pc{
({\Omega_{B/A} 'E})^{\sst} \ar[r] \ar[d] & 
(\ov{E_B})^{\sst} \ar[r] \ar[d] &  (\ov{E_A})^{\sst} \ar[d] \\
(\Omega_{B/A}\bar{E})^{\sst} \ar[r] & (\bar{E}_B)^{\sst} \ar[r] & 
(\bar{E}_A)^{\sst}}
\end{equation}
of presheaves of spectra on $\Sch_{\mathbb{C}}$, where the rows are object-wise fiber sequences. Since each vertical arrow in ~\eqref{eqn:Hom-spec-1} is a morphism of presheaves of spectra on $\Sch_{\C}$ which is an isomorphism on $\Sm_{\C}$, by Theorem~\ref{thm:rec prin} each vertical arrow in \eqref{eqn:Hom-spec-3} is an object-wise weak-equivalence on $\Sm_{\C}$. By definition and Lemma \ref{lem:cor sub}, the map $(\bar{E}_B)^{\sst} \to (\bar{E}^{\sst})_B = (E^{\sst})_B$ is an isomorphism on $\Sm_{\mathbb{C}}$, and the same is true for $E_A$. Composing these with the vertical maps in \eqref{eqn:Hom-spec-3}, and using the identification $(E_B)^{\sst} = (\ov{E_B})^{\sst}$, we get a commutative diagram 
\begin{equation}\label{eqn:Omega cofiber}
\xymatrix{
({\Omega_{B/A}' E})^{\sst} \ar[r] \ar[d] & 
(E_B)^{\sst} \ar[r] \ar[d] &  (E_A)^{\sst} \ar[d] \\
(\Omega_{B/A}\bar{E})^{\sst} \ar[r] & ({E}^{\sst})_B \ar[r] & 
({E}^{\sst})_A}
\end{equation}
of presheaves of spectra on $\Sm_{\mathbb{C}}$, where the rows are object-wise fiber sequences and the vertical arrows are object-wise weak-equivalences in $\Sm_{\C}$. Consider the sequence of maps $\theta \circ u^{\sst}: (\ov{\Omega_{B/A} E})^{\sst} \to (\Omega_{B/A}\bar{E})^{\sst} \to \Omega_{B/A}(\bar{E}^{\sst})$. The map $\theta$ is given by the universal property, and it is an isomorphism since the bottom row of \eqref{eqn:Omega cofiber} is an object-wise fiber sequence. The composite $\theta \circ u^{\sst}$ is an object-wise weak-equivalence on $\Sm_{\mathbb{C}}$ because the map $\ov{\Omega_{B/A} E} \to \Omega_{B/A}\bar{E}$ is an isomorphism on $\Sm_{\C}$ as in \eqref{eqn:Hom-spec-1} so Theorem \ref{thm:rec prin} applies. Thus, by the 2-out-of-3 axiom, the map $u^{\sst}$ is an object-wise weak-equivalence on $\Sm_{\C}$. Since $(\ov{\Omega_{B/A} E})^{\sst} =  ({\Omega_{B/A} E})^{\sst}$ and $\Omega_{B/A}(\bar{E}^{\sst}) = \Omega_{B/A}(E^{\sst})$, we are done.
\end{proof}

\begin{cor}\label{cor:sst of Gm loop-*}
For $E \in \Spt(\C)$, the maps $(\Omega_tE)^{\sst} \to \Omega_{t}(E^{\sst})$ and $(\Omega_{T}E)^{\sst} \to \Omega_{T}(E^{\sst})$ in $\Spt(\C)$ are object-wise weak-equivalences on $\Sm_{\C}$.
\end{cor}

The above follows by applying Proposition \ref{prop:sst of Gm loop} to $t= (\mathbb{G}_m, 1)$ and $T = (\mathbb{P}^1, \infty)$.
Using that $\Omega_{S^1}E(X) = \Omega_{S^1}\left(E(X)\right)$ for a presheaf of spectra $E$ on $\Sch_{\C}$ and that $\Omega_{S^1}\bar{F} \cong \ov{\Omega_{S^1}F}$ for a presheaf of spectra $F$ on $\Sm_{\C}$, one checks that for a presheaf $E$ of spectra on $\Sm_{\mathbb{C}}$, the map $(\Omega_{S^1}E)^{\sst} \to \Omega_{S^1}(E^{\sst})$ is an isomorphism.

\section{Homotopy semi-topologization}\label{sec:sst BG}
We prove that the classes of B.G. and $\A^1$-B.G. presheaves of spectra are closed under semi-topologization. We prove similar results for $(s,\mathfrak{p})$-bispectra. Using these we define \emph{homotopy semi-topologization} on motivic homotopy categories.

\subsection{On $\mathcal{SH}_{S^1}(\C)$} 
For a simplicial spectrum $E$, let $E_p:= E(\Delta[p])$ for $p \geq 0$.

\begin{lem}\label{lem:diag-cofibration}
If each $E_p$ of a simplicial spectrum $E$ is cofibrant, then so is $\diag E$.
\end{lem}
\begin{proof}
By Theorem \ref{thm:MSSpectra}, we need to show that each map $S^1 \wedge (\diag E)_n \to (\diag E)_{n+1}$ in $\Spc_{\bullet}$ is a monomorphism. For a monomorphism $A \to B$ in $\Spc$, one has $(B/A)_n = {B_n}/{A_n}$. For $A, B \in \Spc_{\bullet}$, one has $(A \times B)_n = A_n \times B_n$ and $(A \wedge B)_n = A_n \wedge B_n$. So, it suffices to show  $(S^1)_p \wedge (\diag E)_{n,p} \to (\diag E)_{n+1,p}$ is a monomorphism. But $(\diag E)_{n,p} = E^n_{p,p}$ (\S\ref{subsubsection:Total}) and $(S^1)_i \wedge E^n_{p,i} \to E^{n+1}_{p,i}$ is a monomorphism because each $E_p$ is cofibrant. Thus, the assertion follows. 
\end{proof}

\begin{lem}\label{lem:diag-cofibration*}
If each $f_p: E_p \to F_p$ of a morphism $f:E \to F$ of simplicial spectra is a cofibration of spectra, then the map $\diag E \to \diag F$, i.e. $|E| \to |F|$ is a cofibration.
\end{lem} 
\begin{proof}
A cofibration of spectra is also a level-wise monomorphism in $\Spc_{\bullet}$. So, $f$ is a level-wise monomorphism of bisimplicial sets, and the map $\diag E \to \diag F$ is a level-wise monomorphism in $\Spc$. By \S\ref{subsubsection:NMSpec} and Theorem \ref{thm:MSSpectra}, we need to show that the spectrum ${\diag F}/{\diag E}$ is cofibrant, where $({\diag F}/{\diag E})_n = {(\diag F)_n}/{(\diag E)_n}$. Let $G = F/E$, where $G^n_{p,q} = {F^n_{p,q}}/{E^n_{p.q}}$. Since $(S^1\wedge F^n_p)/{(S^1 \wedge E^n_p)} \simeq S^1\wedge ({F^n_p}/{E^n_p})$, we see that $G$ is a simplicial spectrum. Furthermore, $(\diag G)_{n,p} = G^n_{p,p} = {F^n_{p,p}}/{E^n_{p,p}} = {(\diag F)_{n,p}}/{(\diag E)_{n,p}} = ({\diag F}/{\diag E})_{n,p}.$ Hence $\diag G = {\diag F}/{\diag E}$. Hence by Lemma \ref{lem:diag-cofibration}, it suffices to show that $G_p$ is a cofibrant spectrum. But, $G_p={F_p}/{E_p}$, and that $E_p \to F_p$ is a cofibration implies that ${F_p}/{E_p}$ is cofibrant. 
\end{proof}

Since Nisnevich or motivic cofibrations between presheaves of spectra on $\Sm_S$ are exactly object-wise cofibrations, we deduce the following from Lemma \ref{lem:diag-cofibration*}.

\begin{cor}\label{cor:diag-cofib*}
If each $f_p: E_p \to F_p$ of a morphism $f: E \to F$ of presheaves of simplicial spectra on $\Sch_{\mathbb{C}}$ is an object-wise (Nisnevich, motivic) cofibration of presheaves of spectra, then the map $\diag E \to \diag F$, i.e. $|E| \to |F|$ is an object-wise (Nisnevich, motivic) cofibration.
\end{cor}

\begin{prop}\label{prop:sst-cofiber}
Let $g \circ f: E \to F \to G$ be an object-wise homotopy cofiber sequence of presheaves of spectra on $\Sch_{\C}$. Then $g ^{\sst} \circ f^{\sst} : E^{\sst} \to F^{\sst} \to G^{\sst}$ is an object-wise homotopy cofiber sequence on $\Sm_{\C}$.
\end{prop}

\begin{proof}
 Recall from \cite[\S A 2]{BF} that $g \circ f: E \to F \to G$ is an object-wise homotopy cofiber sequence if and only if we have a sequence $g'\circ f': E \to H \to H/E$ of presheaves of spectra on $\Sch_{\C}$, $h: H \to F$, $p: H/E \to G$, where $f'$ is an object-wise cofibration, $h, p$ are object-wise weak-equivalences, such that $h \circ f' = f$ and $p \circ g' = g \circ h$.
By Theorem \ref{thm:rec prin}, $h^{\sst}$ and $p^{\sst}$ are object-wise weak-equivalences on $\Sm_{\C}$. Using Proposition \ref{prop:sst and wedge}(3), it remains to show that the map $ {f'}^{\sst}: E^{\sst} \to  H^{\sst}$ is an object-wise cofibration, equivalently, that the map $\diag (\tilde{E}) \to \diag(\tilde{H})$ is a cofibration, where $\tilde{E}$ is the presheaf of simplicial spectra on $\Sch_{\C}$ defined by $\tilde{E}(\Delta[p])(-) = E(\Delta^p_{\rm top} \times -)$ (Definition \ref{def:presst}) and similarly for $\tilde{H}$. Since $\tilde{E}(\Delta[p]) \to \tilde{H}(\Delta[p])$ is a filtered colimit of object-wise cofibrations, by \cite[Proposition~3.2]{Mitchel} this map is an object-wise cofibration. Hence, by Corollary \ref{cor:diag-cofib*}, the map $\diag (\tilde{E}) \to \diag (\tilde{H})$ is an object-wise cofibration. This finishes the proof.
\end{proof}

\begin{thm}\label{thm:s resp cd-ext}
Let $E$ be a presheaf of spectra (or complexes of abelian groups) on $\Sm_{\C}$. If $E$ is B.G., then so is $E^{\sst}$. If $E$ is $\mathbb{A}^1$-B.G., then so is $E^{\sst}$.
\end{thm}

\begin{proof}
We prove it for presheaves of spectra for the other is a special case via Dold-Kan correspondence. We prove the first statement. Via the artificial extension in Definition \ref{def:extension}, regard $E$ as a presheaf on $\Sch_{\mathbb{C}}$. Given $X \in \Sm_{\mathbb{C}}$, the presheaf $E_X$ on $\Sch_{\mathbb{C}}$ is $E_X (Y):= E(X \times Y)$ for $Y \in \Sch_{\mathbb{C}}$. Given a Nisnevich square as in \eqref{eqn:elementary square}, where $X, U, V, W \in \Sm_{\C}$ with $W= U \times_X V$, we have a commutative diagram 
\begin{equation}\label{eqn:sst diagram middle o}
\xymatrix{
E_X \ar[r]^{j_1} \ar[d]_{j_2} & E_U \ar[d]^{h_1}\\
E_V \ar[r]^{h_2} & E_W}
\end{equation}
of presheaves of spectra on $\Sch_{\mathbb{C}}$, which is object-wise homotopy Cartesian on $\Sm_{\mathbb{C}}$. Equivalently, it is object-wise homotopy co-Cartesian on $\Sm_{\C}$. Let $G_1$ and $G_2$ be the object-wise homotopy cofibers of $j_1$ and $h_2$. Then \eqref{eqn:sst diagram middle o} is object-wise homotopy co-Cartesian on $\Sm_{\C}$ if and only if the map $h: G_1 \to G_2$ is an object-wise weak-equivalence on $\Sm_{\C}$. So, by Theorem \ref{thm:rec prin}, the map $ h^{\sst}: G^{\sst}_1 \to  G^{\sst}_2$ is an object-wise weak-equivalence on $\Sm_{\C}$. Using Proposition \ref{prop:sst-cofiber}, we obtain a commutative diagram
\begin{equation}\label{eqn:sst diagram middle a}
\xymatrix{
E^{\sst}_X \ar[r]^{j^{\sst}_1} \ar[d]_{j^{\sst}_2} & E^{\sst}_U \ar[d]^{h^{\sst}_1} \ar[r] & G^{\sst}_1 \ar[d]^{h^{\sst}} \\
E^{\sst}_V \ar[r]^{h^{\sst}_2} & E^{\sst}_W \ar[r] & G^{\sst}_2,}
\end{equation}
where the rows are object-wise homotopy cofiber sequences of presheaves on $\Sm_{\C}$. Since $h^{\sst}$ is an object-wise weak-equivalence on $\Sm_{\C}$, the left square in \eqref{eqn:sst diagram middle a} is object-wise homotopy co-Cartesian on $\Sm_{\C}$. Equivalently, it is object-wise homotopy Cartesian on $\Sm_{\C}$. Evaluating at $\Spec(\C)$ and applying Lemma \ref{lem:cor sub}, we obtain 
$$\xymatrix{
E^{\sst}(X) \ar[r]^{j^{\sst}_1} \ar[d]_{j^{\sst}_2} & E^{\sst}(U) \ar[d]^{h^{\sst}_1} \\
E^{\sst}(V) \ar[r]^{h^{\sst}_2} & E^{\sst}(W),}
$$
a homotopy Cartesian square of spectra. This shows that $E^{\sst}$ is B.G. as desired. The second statement follows from the first and Proposition \ref{prop:spt basic sst}. 
\end{proof}

Applying Theorems \ref{thm:A1Nis descent}, \ref{thm:rec prin} and \ref{thm:s resp cd-ext}, we conclude:

\begin{cor}\label{cor:BGA1 sst}
The $\sst$ of an $S^1$-stable motivic weak-equivalence of $\mathbb{A}^1$-B.G. motivic spectra is an object-wise weak-equivalence of $\mathbb{A}^1$-B.G. motivic spectra.
\end{cor}

\begin{cor}\label{cor:sst hat 1}
There exists an endo-functor $\hosst : \mathcal{SH}_{S^1} (\mathbb{C}) \to \mathcal{SH}_{S^1} (\mathbb{C})$, which coincides with the $\sst$-functor on $\mathbb{A}^1$-B.G. motivic spectra up to isomorphism. 
\end{cor}
\begin{proof}We know from Theorem \ref{thm:s resp cd-ext} that $\sst: \Spt(\C) \to \Spt(\C)$ preserves $\A^1$-B.G. motivic spectra. 
Since an $S^1$-stable motivic fibrant motivic spectrum is Nisnevich fibrant and $\A^1$-local, it is $\A^1$-B.G. by Lemma \ref{lem:BG A1wi A1loc}. By Corollary \ref{cor:BGA1 sst}, we know $\sst$ takes a trivial motivic fibration between $S^1$-stable motivic fibrant motivic spectra into an $S^1$-stable motivic weak-equivalence. Thus, by  \cite[Proposition~8.4.8]{Hirschhorn} we obtain a right derived endo-functor $\hosst : \mathcal{SH}_{S^1} (\mathbb{C}) \to \mathcal{SH}_{S^1} (\mathbb{C})$, with desired properties. 
\end{proof}

\subsection{On $\mathcal{SH}(\C)$}\label{sec:sst T}

Let $E= (E_0, E_1, \cdots) \in \Spt_{(s,\mathfrak{p})}(\C)$ (\S \ref{subsubsection: bispec} and \S \ref{subsubsection:T-spec-M}) with the bonding maps $ T \wedge E_n \to E_{n+1}$. It yields $ \left(T \wedge E_n\right)^{\sst} = T^{\sst} \wedge E^{\sst}_n \to E^{\sst}_{n+1}$, by Proposition \ref{prop:sst and wedge}. Composed with $T\wedge E^{\sst}_n \to T^{\sst} \wedge E^{\sst}_n$, we get $  T \wedge E^{\sst}_n \to  E^{\sst}_{n+1}$. This gives $E^{\sst}: = (E^{\sst}_0, E^{\sst}_1, \cdots) \in \Spt_{(s,\mathfrak{p})}(\C)$. One checks  ${\rm Id} \to ( -)^{\sst}$ is natural on $\Spt_{(s, \mathfrak{p})}(\C)$.

\begin{thm}\label{thm:BGA1 sst-bispec}
$(1)$ The class of $\A^1$-B.G. $(s,\mathfrak{p})$-bispectra is closed under the $\sst$-functor. $(2)$ The class of $\A^1$-B.G. motivic $\Omega_T$-bispectra (Definition \ref{defn:BG bispec}) is closed under the $\sst$-functor. $(3)$ If $f$ is a stable motivic weak-equivalence of $\mathbb{A}^1$-B.G. motivic $\Omega_T$-bispectra, then $f^{\sst}$ is a $T$-level-wise object-wise weak-equivalence of $\mathbb{A}^1$-B.G. motivic $\Omega_T$-bispectra. 
\end{thm}
\begin{proof}
Part (1) holds by Theorem \ref{thm:s resp cd-ext}. For Part (2), let $E$ be an $\A^1$-B.G. motivic $\Omega_T$-bispectrum. Using Lemma \ref{lem:A1BGOmega-Spec 1} we deduce that each $\Omega_TE_n$ is an $\A^1$-B.G. motivic $S^1$-spectrum. So by Theorem \ref{thm:A1Nis descent} and Corollary \ref{cor:BGA1 sst}, the map $E^{\sst}_n \to (\Omega_TE_{n+1})^{\sst}$ is an object-wise weak-equivalence. Now by Corollary \ref{cor:sst of Gm loop-*}, the map $(E^{\sst})_n \to \Omega_T((E^{\sst})_{n+1})$ is an object-wise weak-equivalence, thus an $S^1$-stable motivic weak-equivalence. Part (3) follows from (1), (2) and Theorems \ref{thm:descent for bispectrum} and  \ref{thm:rec prin}.
\end{proof}

Recall that for a morphism $f: E \to F$ in $\Spt_{(s,\mathfrak{p})}(\C)$, the cone $C(f)$ is the push-out 
\begin{equation}\label{eqn:Cone}
\xymatrix{
E \ar[r]^<<<<<<{{\rm Id} \wedge 0} \ar[d]^{f} & E \wedge \Delta[1]
\ar[d]^{\bar{f}} \\
F \ar[r]^{\tilde{f}} & C(f),}
\end{equation}
where $\Delta[1]$ is pointed by one. Collapsing $F$ to the base point of $\Sigma_sE = E \wedge S^1$ and using the quotient map $E \wedge \Delta[1] \to E \wedge S^1$, we get $\delta_f: C(f) \to \Sigma_sE$, which gives
$\delta_f\circ \tilde{f} \circ f : E \to F \to  C(f) \to \Sigma_sE.$

\begin{lem}\label{lem:sst-tr}
Let $f: E \to F$ be a morphism in $\Spt_{(s,\mathfrak{p})}(\C)$. Then, the following is a push-out square:
\begin{equation}\label{eqn:push-out stuff}
\xymatrix{
E^{\sst} \ar[r]^<<<<<<{{\rm Id} \wedge 0} \ar[d]_{f^{\sst}} & E^{\sst} \wedge 
\Delta[1] \ar[d]^{{\bar{f}}^{\sst}} \\
F^{\sst} \ar[r]^{{\wt{f}}^{\sst}} & (C(f))^{\sst}.}
\end{equation}
\end{lem}
\begin{proof}
For a presheaf $G$ on $\Sm_{\C}$ of objects in $\Spc_{\bullet}$ or $\Spt$. Let $\bar{G}$ be its artificial extension on $\Sch_{\C}$ as in Definition \ref{def:extension}. This extends $(s,\mathfrak{p})$-bispectra over $\Sm_{\mathbb{C}}$ to $(s,\mathfrak{p})$-bispectra over $\Sch_{\mathbb{C}}$. Note the push-out of a diagram of presheaves of $(s,\mathfrak{p})$-bispectra is defined object-wise, and one has $\ov{E \wedge \Delta[1]} \simeq \bar{E} \wedge \Delta[1]$. 
Since the artificial extension is defined as a colimit and since the colimits commute among themselves (\emph{cf.} \cite[\S~IX.8]{ML}), the push-out \eqref{eqn:Cone} (a colimit) remains a push-out square if we replace the presheaves by their artificial extensions. So, we may assume the presheaves $E$ and $F$ are defined on $\Sch_{\C}$. The commutativity of two colimits also implies that the diagram  
\begin{equation}\label{eqn:Cone-1}
\xymatrix{
E(\Delta^n_{\rm top} \times -) \ar[r]^<<<<<<{{\rm Id} \wedge 0} \ar[d]_{f} & 
E(\Delta^n_{\rm top} \times -) \wedge \Delta[1] 
\ar[d]^{\bar{f}} \\
F(\Delta^n_{\rm top} \times -) \ar[r]^{\wt{f}} & C(f)(\Delta^n_{\rm top} \times -)}
\end{equation}
is a push-out square. Since
$\Hom_{\sC}(X \wedge \Delta[k]_{+}, Y ) \simeq \Hom_{\sC}(X, {\mathcal{H}om}_{\bullet}(\Delta[k]_{+}, Y))$
where $\sC$ is the category of $(s,\mathfrak{p})$-bispectra on $\Sch_{\C}$, we deduce that \eqref{eqn:Cone-1} remains a push-out square after smashing with $\Delta[k]_{+}$ for $k \ge 0$. Since a coequalizer (a colimit) commutes with colimits, by \S\ref{subsubsection:Total} we obtain a push-out square \eqref{eqn:push-out stuff} except $E^{\sst} \wedge \Delta[1]$ is replaced with $(E \wedge \Delta[1])^{\sst}$. But, by Proposition \ref{prop:sst and wedge}(4) and the isomorphism $\Delta[1] \simeq (\Delta[1])^{\sst}$, we do have $E^{\sst} \wedge \Delta[1] \simeq (E \wedge \Delta[1])^{\sst}$.
\end{proof}

\begin{thm}\label{thm:sst hat-bispec}
There exists a triangulated endo-functor $\hosst : \mathcal{SH} (\mathbb{C}) \to \mathcal{SH} (\mathbb{C})$, which coincides with the $\sst$-functor on $\mathbb{A}^1$-B.G. motivic $\Omega_T$-bispectra up to isomorphism. 
\end{thm}
\begin{proof}
By Theorem \ref{thm:BGA1 sst-bispec}, we know $\A^1$-B.G. motivic $\Omega_T$-bispectra are closed under $\sst$. By \cite[Lemma~2.3.8]{Morel}, the functor $\Sigma_T: \Spt(\C) \to \Spt(\C)$ preserves stable motivic weak-equivalences and cofibrations. Hence, $\Sigma_T$ is a left Quillen endo-functor with the right adjoint $\Omega_T: \Spt(\C) \to \Spt(\C)$. An $(s,\mathfrak{p})$-bispectrum $E = (E_0, E_1, \cdots )$ is stable motivic fibrant if and only if it is a motivic $\Omega_T$-bispectrum and it is $T$-level-wise $S^1$-stable motivic fibrant (\emph{cf.} \cite[Definition 3.1, Theorem 3.4]{Hovey}). So, a stable motivic fibrant $(s,\mathfrak{p})$-bispectrum is an $\mathbb{A}^1$-B.G. motivic $\Omega_T$-bispectrum. By Theorem \ref{thm:BGA1 sst-bispec}, we know $\sst$ takes a trivial stable motivic fibration between stable motivic fibrant $(s,\mathfrak{p})$-bispectra to a stable motivic weak-equivalence. Thus, by \cite[Proposition~8.4.8]{Hirschhorn} we obtain a right derived endo-functor $\hosst : \mathcal{SH} (\mathbb{C}) \to \mathcal{SH} (\mathbb{C})$ with desired properties.  
We now check that $\hosst : \mathcal{SH} (\mathbb{C}) \to \mathcal{SH} (\mathbb{C})$ is a triangulated functor. Since $\hosst$ preserves finite coproducts and products in $\mathcal{SH} (\mathbb{C})$, it is an additive functor. The shift $E \mapsto E[1]$ on $ \mathcal{SH} (\mathbb{C})$ is given by the functor $E \mapsto \Sigma_sE$. One sees that $\hosst$ commutes with $\Sigma_s$ by Proposition \ref{prop:sst and wedge} and the isomorphism $S^1 \simeq (S^1)^{\sst}$. For a distinguished triangle in $\mathcal{SH}(\C)$ of the form $E \to F \to C(f) \to \Sigma_s E$ for a map $f: E \to F$ in $\Spt_{(s,\mathfrak{p})}(\C)$ (\emph{cf.} \cite[\S 2.3]{Voevodsky Norway}), by Lemma \ref{lem:sst-tr} and the isomorphism $(\Sigma_sE)^{\sst} \simeq \Sigma_s E^{\sst}$, we deduce that $E^{\sst} \to  F^{\sst} \to (C(f))^{\sst} \to \Sigma_sE^{\sst}$ is also a distinguished triangle in $\mathcal{SH}(\C)$.
\end{proof}

\begin{defn}
For the rest of this paper, we call the functor $\hosst$ of Corollary \ref{cor:sst hat 1} and Theorem \ref{thm:sst hat-bispec} by the name \emph{homotopy semi-topologization} functor. For any $E$ in $ \mathcal{SH}_{S^1}(\mathbb{C})$ or $\mathcal{SH} (\mathbb{C})$, we denote $\hosst(E)$ by $E^{\hosst}$. 
\end{defn}

\section{Representing semi-topological $K$-theory in $\mathcal{SH} (\mathbb{C})$}\label{sec:representing-K}

We prove that the semi-topological $K$-theory of \cite{FW0} is representable in motivic homotopy categories. For $\mathcal{SH}_{S^1} (\C)$, it is easy by semi-topologizing an $\mathbb{A}^1$-B.G. presheaf of spectra representing the algebraic $K$-theory. For $\mathcal{SH}(\C)$, an essence is to find an $\mathbb{A}^1$-B.G. motivic $\Omega_T$-bispectrum that represents algebraic $K$-theory. See Proposition \ref{prop:A1BF K-alg}. 

The semi-topological $K$-theory of a complex variety $X$ is a bridge between the algebraic and the topological $K$-theories of $X$. 
This theory was defined in \cite{FW1} as the stable homotopy groups of an infinite loop space, constructed out of the stabilization of the analytic space of algebraic morphisms of complex varieties. In \cite{FW0}, another definition of the semi-topological $K$-theory is given by
$K^{\sst}_p(X) :=  \pi_p (|\sK (\Delta^{\bullet}_{\rm top} \times X ) | )$
for $p \in \Z$, where $\sK(-)$ is the presheaf of connective spectra on $\Sch_{\C}$ that represents Quillen algebraic $K$-theory. By \cite[Theorem~1.4]{FW0}, this definition coincides with the original one in \cite{FW1} for projective weakly normal varieties.

\subsection{Representability in $\mathcal{SH}_{S^1}(\C)$}\label{subsection:Alg-K}

Recall that \cite{Jardine K} (see also \cite{Kim}) constructed a presheaf of spectra on $\Sm_{\mathbb{C}}$ that represents the algebraic $K$-theory. This construction and some properties are summarized as follows, taken from {\cite[Theorem~5, Proposition~9]{Jardine K}, and \cite[Proposition 6.8, Theorem 10.8]{TT}}: 

\begin{thm}\label{thm:Jardine*}
There is a presheaf $\mathcal{K}$ of spectra on $\Sm_{\C}$ such that for $X \in \Sm_{\C}$, $\sK(X)$ represents the algebraic $K$-theory of $X$. This is a presheaf of $\Omega_s$-spectra above level zero, equipped with smash product $\sK_i \wedge \sK_j \to \sK_{i+j}$ which commutes with the bonding maps of $\sK$. Furthermore, $\sK$ is an $\A^1$-B.G. presheaf of spectra on $\Sm_{\C}$. 
\end{thm}

For representability of semi-topological $K$-theory in $\mathcal{SH}_{S^1}(\C)$ (and $\mathcal{H}_{\bullet}(\C)$), we have a quick answer. Let $\sK$ be presheaf of spectra on $\Sm_{\C}$ as in Theorem \ref{thm:Jardine*}.

\begin{prop}\label{prop:K rep s-mot}
Let $X \in \Sm_{\C}$ and $p \in \Z$. Then, we have $K^{\sst}_p(X) \simeq [\Sigma^{\infty}_sX_{+}[p], \sK^{\sst}]_{\A^1}.$
That is, the semi-topological $K$-theory is representable in $\mathcal{SH}  _{S^1}(\mathbb{C})$.
\end{prop} 
\begin{proof}
It holds by Corollary \ref{cor:Descent-CR-1}, Theorems \ref{thm:s resp cd-ext}, \ref{thm:Jardine*}, and the definition of $K_p ^{\sst}$.
\end{proof}

\begin{cor}\label{cor:K rep uns-mot}
For $X \in \Sm_{\C}$ and $p \ge 0$, we have $K^{\sst}_p(X) \simeq [\Sigma_s^p X_{+}, {\bf R}{Ev}_0\sK^{\sst}]_{\A^1}.$
That is, the semi-topological $K$-theory is representable in $\mathcal{H}_{\bullet} (\mathbb{C})$.
\end{cor}

\subsection{Representability in $\mathcal{SH}(\C)$}
\label{subsection:Alg-K-SP} 
For a presheaf of spectra $E = \left(E_0, E_1, \cdots \right)$ on $\Sm_{\C}$, let $E\{n\}$ be the presheaf of spectra $\left(E_n, E_{n+1}, \cdots \right)$. We use $\mathcal{K}$ of \S \ref{subsection:Alg-K} in what follows. Let $f: \sK \to \sK^{\rm fib} \leftarrow \sK^{\rm cf} : g $ be two morphisms in $\Spt(\C)$, where $f$ is an $S^1$-stable motivic fibrant replacement of $\sK$ and $g$ is an $S^1$-stable motivic cofibrant replacement of $\sK^{\rm fib}$. Since $\sK^{\rm fib}$ is motivic fibrant and $g$ is a motivic fibration, it follows that $\sK^{\rm cf}$ is motivic cofibrant-fibrant. Moreover, by Theorem \ref{thm:MSSpectra} and Corollary \ref{cor:A1BG consequence}, each $\sK^{\rm cf}\{n\}$ is motivic cofibrant-fibrant. By Theorem \ref{thm:Jardine*} and Corollary \ref{cor:A1BG consequence}, the maps $\sK_n \rightarrow \sK^{\rm fib}_n  \leftarrow \sK^{\rm cf}_n$ are object-wise weak-equivalences for each $n \ge 1$. Using the product structure on $\sK$ in Theorem \ref{thm:Jardine*}, we obtain a morphism of motivic spectra $\sK^{\rm cf}\wedge \sK^{\rm cf}_1 \to \sK^{\rm cf}\{1\}$ in $\mathcal{SH}_{S^1} (\C)$. This is equivalent to a morphism $\sK^{\rm cf} \to {\bf R}\Omega_{\sK^{\rm cf}_1} \sK^{\rm cf}\{1\} \simeq \Omega_{\sK^{\rm cf}_1} \sK^{\rm cf}\{1\}$ in $\mathcal{SH}_{S^1} (\C)$. Since $\sK^{\rm cf}$ is cofibrant and $\Omega_{\sK^{\rm cf}_1} \sK^{\rm cf}\{1\}$ is fibrant, this map lifts to a map in $\Spt(\C)$. Taking the adjoint of this map, we conclude that there is a morphism
$\phi: \sK^{\rm cf}\wedge \sK^{\rm cf}_1 \to \sK^{\rm cf}\{1\}$ in $\Spt(\C)$.
Thus, we obtained a cofibrant-fibrant motivic spectrum model $\sK^{\rm cf}$ for the algebraic $K$-theory, with a product that yields a ring structure on $K_* (X)$ for $X \in \Sm_{\mathbb{C}}$. 
The above product structure on the presheaf $\sK^{\rm cf}$  of spectra allows one to construct a $T$-spectrum that represents the algebraic $K$-theory in $\mathcal{SH}(\mathbb{C})$. For details, we refer to \cite{Kim}. To prove the representability of the semi-topological $K$-theory in $\mathcal{SH}(\mathbb{C})$, we lift this $T$-spectrum to an $(s,\mathfrak{p})$-bispectrum, for which we recycle the construction in \cite[\S6.2]{Voevodsky}.

\begin{lem}\label{lem:Thom-1}
Let $X \in \Sm_{\C}$. $(1)$ For $p \ge m \ge 0$, we have $[\Sigma_s^{p} X_{+}, \sK^{\rm cf}_m]_{\A^1} \simeq K_{p-m}(X),$
and a split exact sequence $0 \to [\Sigma_s^{p} \Sigma_T X_{+}, \sK^{\rm cf}_m]_{\A^1} \to K_{p-m}(\P^1_X) \to  K_{p-m}(X) \to 0.$ 

 $(2)$ For $0 \le p < m$, $[\Sigma_s^{p} X_{+}, \sK^{\rm cf}_m]_{\A^1} = [\Sigma_s^{p} \Sigma_T X_{+}, \sK^{\rm cf}_m]_{\A^1} = 0$ and there is a split exact sequence $0 \to [\Sigma_s^{p} \Sigma_T X_{+}, \sK^{\rm cf}_m]_{\A^1} \to K_{p-m}(\P^1_X) \to  K_{p-m}(X) \to 0.$
\end{lem}

\begin{proof}
For $p \ge 0$, the cofiber sequence $\Sigma^{\infty}_s \Sigma_s^p X_{+} \to \Sigma^{\infty}_s \Sigma_s^p  (\P^1_X)_{+} \to \Sigma^{\infty}_s \Sigma_s^p \Sigma_T X_{+}$ in $\mathcal{SH}_{S^1} (\mathbb{C})$ (\emph{cf.} \cite[Lemma~2.16]{Voevodsky Norway}) and Lemma \ref{lem:A1Nis-fib pi} give us a long exact sequence
$ \to [\Sigma^{\infty}_s \Sigma_s^{p} \Sigma_T X_{+}, \sK^{\rm cf}]_{\A^1} \to K_p(\P^1_X) \to K_p(X) \to,$
where the map $i^*_0: K_p (\P^1_X) \to K_p (X)$ splits by the pull-back via the projection $X \times \P^1 \to X$. Part (1) follows easily from this and the adjoint isomorphisms
$[\Sigma^{\infty}_s A, \sK^{\rm cf}]_{\A^1} \simeq [A, \sK^{\rm cf}_0]_{\A^1} \simeq  [A, \Omega^m_s\sK^{\rm cf}_m]_{\A^1} \simeq  [\Sigma_s^m A, \sK^{\rm cf}_m]_{\A^1}$
for $A \in \Spc_{\bullet}(\C)$. Notice here that $\sK^{\rm cf}$ and $\sK^{\rm cf}_m$ are all motivic (hence object-wise) fibrant
and $\sK^{\rm cf}$ is a motivic $\Omega_s$-spectrum. To prove the first part of $(2)$, first use Lemma \ref{lem:A1Nis-fib pi} and Corollary \ref{cor:A1BG consequence} to obtain isomorphisms
$[\Sigma_s^p X_{+}, \sK^{\rm cf}_m]_{\A^1}  \simeq  \pi_p (\sK^{\rm cf}_m(X) ) \simeq  \pi_{p-m} (\sK^{\rm cf}(X)),$
where the last term is zero if $p-m < 0$ since $\sK^{\rm cf}(X)$ is a connective spectrum. For the second part of $(2)$, from the cofiber sequence $\Sigma^{\infty}_s \Sigma_s^p X_{+} \to \Sigma^{\infty}_s \Sigma_s^p  (\P^1_X)_{+} \to \Sigma^{\infty}_s \Sigma_s^p \Sigma_T X_{+}$, we get
an exact sequence $[\Sigma^{\infty}_s \Sigma_s^{p+1} (\P^1_X)_{+}, \sK^{\rm cf}\{m\}]_{\A^1} \to 
[\Sigma^{\infty}_s \Sigma_s^{p+1} X_{+}, \sK^{\rm cf}\{m\}]_{\A^1} \to 
[\Sigma^{\infty}_s  \Sigma_s^{p} \Sigma_T X_{+}, \sK^{\rm cf}\{m\}]_{\A^1}
\to
[\Sigma^{\infty}_s \Sigma_s^{p} (\P^1_X)_{+}, \sK^{\rm cf}\{m\}]_{\A^1}.$
By Corollary \ref{cor:A1BG consequence} and the adjointness, this exact sequence is equivalent to
$[\Sigma_s^{p+1} (\P^1_X)_{+}, \sK^{\rm cf}_m]_{\A^1} \to [\Sigma_s^{p+1} X_{+}, \sK^{\rm cf}_m]_{\A^1} \to [\Sigma_s^{p} \Sigma_T \wedge X_{+}, \sK^{\rm cf}_m]_{\A^1} \to [\Sigma_s^{p} (\P^1_X)_{+}, \sK^{\rm cf}_m]_{\A^1}.$ It follows from $(1)$ and the first part of $(2)$ that the first map in this exact sequence is surjective and the last term is zero if $0 \le p < m$. Hence the third term must be zero.
\end{proof}

Recall the ring isomorphism ${K_0(\C)[t]}/{(t-1)^2} \simeq  K_0(\P^1_{\C})$. By Lemma \ref{lem:Thom-1}, the element $(t-1) = \left([\sO(1)] -[\sO]\right)$ defines a unique element $th \in  [S^1 \wedge T, \sK^{\rm cf}_1]_{\A^1}$, called the \emph{Thom class}. Since $S^1 \wedge T$ is cofibrant and $\sK^{\rm cf}_1$ is motivic fibrant, this yields a morphism $ S^1\wedge T \to \sK^{\rm cf}_1$ in $\Spc_{\bullet} (\C)$, thus a morphism $\theta : T \to \Omega_s \sK^{\rm cf}_1$ in $\Spc_{\bullet}(\C)$.

\begin{defn}\label{defn:SP-bispec-K}
Define $\sK^{\rm alg} = \{\sK^{\rm alg}_{m,n} \} \in \Spt_{(s,\mathfrak{p})}(\C)$ as $ ( \sK^{\rm cf}, \Omega^1_s \sK^{\rm cf}\{1\}, \Omega^2_s \sK^{\rm cf}\{2\}, \cdots )$ with the following bonding maps: for $A\in \Spt(\C), B \in \Spc_{\bullet}(\C)$, apply the map $\Omega_sA \wedge B \to \Omega_s(A \wedge B)$ repeatedly to get the morphisms $\Omega^n_s \sK^{\rm cf}\{n\} \wedge T \to \Omega^n_s (\sK^{\rm cf}\{n\} \wedge T )  \xrightarrow{\Omega^n_s({\rm Id} \wedge \theta)} \Omega^n_s (\sK^{\rm cf}\{n\} \wedge \Omega_s \sK^{\rm cf}_1 ) \to \Omega^{n+1}_s (\sK^{\rm cf}\{n\} \wedge \sK^{\rm cf}_1 ) \xrightarrow{\Omega^{n+1}_s({\rm Id} \wedge \phi)} \Omega^{n+1}_s\sK^{\rm cf}\{n+1\}.$
\end{defn}

\begin{prop}\label{prop:A1BF K-alg}
The $(s,\mathfrak{p})$-bispectrum $\sK^{\rm alg}$ on $\Sm_{\C}$ is an $\A^1$-B.G. motivic $\Omega_T$-bispectrum, and it represents the algebraic $K$-theory in $\mathcal{SH} (\mathbb{C})$.
\end{prop}
\begin{proof}
Since $\sK^{\rm alg}_{*,n} =  \Omega^n_s \sK^{\rm cf}\{n\}$ for each $n \ge 0$, by Corollary \ref{cor:A1BG consequence} we see $\sK^{\rm alg}$ satisfies the $\A^1$-B.G. property. To show that $\sK^{\rm alg}$ is a motivic $\Omega_T$-bispectrum, it suffices to show that each map $\sK^{\rm alg}_{m,n} \to \Omega_T \sK^{\rm alg}_{m, n+1}$ between motivic fibrant pointed motivic spaces is a motivic weak-equivalence. For this, it suffices to show using Corollary \ref{cor:A1BG consequence} and Lemma \ref{lem:A1Nis-fib pi} that for $X \in \Sm_{\C}$ and $p \ge 0$, the induced map $[\Sigma_s^p X_{+}, \Omega^n_s \sK^{\rm cf}_{m+n}]_{\A^1} \to [\Sigma_s^p X_{+}, \Omega_T \Omega^{n+1}_s \sK^{\rm cf}_{m+n+1}]_{\A^1}$ is an isomorphism, or that the map $[\Sigma_s^p X_{+}, \Omega^n_s \sK^{\rm cf}_{m+n}]_{\A^1} \to [\Sigma_s^p \Sigma_T X_{+}, \Omega^{n+1}_s \sK^{\rm cf}_{m+n+1}]_{\A^1}$ is an isomorphism. Using Corollary \ref{cor:A1BG consequence}, this is equivalent to that
$[\Sigma_s^{p} X_{+}, \sK^{\rm cf}_{m}]_{\A^1} \to [\Sigma_s^{p} \Sigma_T X_{+}, \sK^{\rm cf}_{m}]_{\A^1}$
is an isomorphism. But, by Lemma \ref{lem:Thom-1} and the definition of the Thom class, for $0 \le p < m$ both terms are zero, while for $p \ge m \ge 0$ this map is just the multiplication by the Thom class on the groups $K_{p-m}(X) \to K_{p-m}(\P^1_X, \{\infty\} \times X).$ This is an isomorphism by the projective bundle formula. The representability now follows using Theorems \ref{thm:descent for bispectrum}, \ref{thm:Jardine*} and Corollary \ref{cor:Descent-CR-1}.
\end{proof}


\begin{thm}\label{thm:K rep mot}
Let $X \in \Sm_{\C}$ and $p \in \Z$. We have
$K^{\sst}_p(X) \simeq \left[\Sigma^{\infty}_T \Sigma^{\infty}_s X_{+}[p], (\sK^{\rm alg})^{\hosst}\right]_{\A^1},
$
i.e. the semi-topological $K$-theory is representable in $\mathcal{SH} (\mathbb{C})$.
\end{thm}
\begin{proof}
By Theorems \ref{thm:BGA1 sst-bispec}, \ref{thm:sst hat-bispec} and  Proposition \ref{prop:A1BF K-alg}, we may replace $(\sK^{\rm alg})^{\hosst}$ in the above by $(\sK^{\rm alg})^{\sst}$. Let $f: (\sK^{\rm alg})^{\sst} \to F$ be a stable motivic fibrant replacement of $(\sK^{\rm alg})^{\sst}$. This $F = (F_0, F_1, \cdots)$ is a $T$-level-wise motivic fibrant motivic $\Omega_T$-bispectrum. We have isomorphisms $[\Sigma^{\infty}_T \Sigma^{\infty}_s X_{+}[p], (\sK^{\rm alg})^{\sst}]_{\A^1}\simeq [\Sigma^{\infty}_T \Sigma^{\infty}_s X_{+}[p], F ]_{\A^1} \simeq [\Sigma^{\infty}_s X_{+}[p], \Omega^{\infty}_T F]_{\A^1} = [\Sigma^{\infty}_s X_{+}[p], F_0]_{\A^1}$. By Theorems  \ref{thm:descent for bispectrum}, \ref{thm:BGA1 sst-bispec}, and Proposition \ref{prop:A1BF K-alg}, the map $f$ is a $T$-level-wise object-wise weak-equivalence. In particular, the map $(\sK^{\rm alg}_{*,0})^{\sst} = (\sK^{\rm alg})^{\sst}_{*,0} \to F_0$ is an object-wise weak-equivalence, so that we have $ [\Sigma^{\infty}_s X_{+}[p], F_0]_{\A^1} \simeq  [\Sigma^{\infty}_s X_{+}[p], (\sK^{\rm alg}_{*,0})^{\sst} ]_{\A^1}\simeq [\Sigma^{\infty}_s X_{+}[p], (\sK^{\rm cf})^{\sst}]_{\A^1}.$ By Theorems \ref{thm:A1Nis descent}, \ref{thm:rec prin}, and \ref{thm:Jardine*}, the maps $\sK^{\sst} \rightarrow (\sK^{\rm fib})^{\sst} \leftarrow (\sK^{\rm cf})^{\sst}$ are object-wise weak-equivalences, so the last group is $[\Sigma^{\infty}_s X_{+}[p], \sK^{\sst}]_{\A^1}.$ But, by Proposition \ref{prop:K rep s-mot}, this is $K^{\sst}_p(X)$.
\end{proof}

\section{Representing morphic cohomology in $\mathcal{SH} (\mathbb{C})$}\label{sec:representing-M}
The morphic cohomology $L^p H^q (X)$ for smooth quasi-projective schemes $X$ over $\C$ was introduced by \cite{FL}, as the homotopy groups of a function space. Later it was identified in \cite{FW2} as the homotopy group of the semi-topologization of the complex of Friedlander and Suslin. We show that the morphic cohomology is representable in $\mathcal{SH} (\mathbb{C})$ by homotopy semi-topologizing the motivic Eilenberg-MacLane spectrum of Voevodsky.

\subsection{Motivic Eilenberg-MacLane spectrum}\label{sec:MEMS}

Recall (\cite[p. 141]{FV}, {\cite[p. 126]{Voevodsky lecture}}) the following. Let $r \geq 0$ and let $f:Z \to U$ be a morphism, where each irreducible component of $Z$ dominates a component of $U$. We say \emph{$Z$ is equidimensional of relative dimension $r$ over $U$} if for every $s \in U$, the scheme-theoretic fiber $Z_s$ is either $ \emptyset$, or an equidimensional of dimension $r$. For $X \in \Sch_{\C}$ and $U \in \Sm_{\C}$, let $z_{\equi} (X, r)(U)$ be the group of cycles on $Z$ of $X \times U$ that are dominant and equidimensional of relative dimension $r$ over a component of $U$. This $z_{\equi}(X,r)$ is a presheaf (in fact an \'etale sheaf) on $\Sm_{\C}$. Let $\Delta^{\bullet}$ be the cosimplicial scheme, where $\Delta^n  = \Spec (\C[t_0, \cdots, t_n]) / (\sum_{i=0} ^n t_i -1)$, and $\partial_i ^n$ ($ 0 \leq i \leq n$) are the cofaces. For $U \in \Sm_{\C}$, and a presheaf $F$ of abelian groups on $\Sm_{\C}$, the simplicial abelian group $F(\Delta^{\bullet} \times U)$ has its associated chain complex $\un{C}_* F (U)$, namely, $\un{C}_n F(U) = F( \Delta^n \times U)$ with the differential $\sum_{i=0} ^n (-1)^i F(\partial_i ^n \times {\rm Id}_U)$. This $\un{C}_* F$ is a presheaf of chain complexes of abelian groups on $\Sm_{\C}$. For $n \geq 0$, the \emph{Friedlander-Suslin motivic complex} $\mathbb{Z}^{FS} (n) $ on $\Sm_{\C}$ is $ \un{C}_* z_{\equi} (\mathbb{A}^n, 0)$. (This definition of $\mathbb{Z}^{FS}(n)$ differs slightly from the one in \cite{Voevodsky lecture}, where $\mathbb{Z}^{FS} (n)$ is defined as $\un{C}_* z_{\equi} (\mathbb{A}^n, 0)[-2n]$.)
In what follows, we identify the presheaf $\mathbb{Z}^{FS} (n)$ with an object of $\Spc_{\bullet}(\C)$ via the Dold-Kan correspondence. Recall (\cite[\S 6.1]{Voevodsky}) that the motivic Eilenberg-MacLane spectrum $\mathbf{H}\mathbb{Z}$ is a sequence of pointed simplicial presheaves, whose $n$-th level is $K(\mathbb{Z}(n), 2n) = \un{C}_* L (T^n)$ for some functor $L$, with motivic weak-equivalences $K(\mathbb{Z} (n), 2n) \to \Omega_{T} K(\mathbb{Z} (n+1), 2n+2)$.  
For $X \in \Sm_{\C}$ and $U \in \Sm_{\C}$, $L(X) (U)$ is the group of cycles on $U \times X$, finite over $U$ and surjective over a connected component of $U$. This $L(X)$ is a presheaf on $\Sm_{\C}$. This $L$ even extends to $\Spc_{\bullet} (\C)$. Using the isomorphisms $T^n \simeq {\P^n}/{\P^{n-1}}$ and $\un{C}_* L(A/B) \simeq {\un{C}_* L(A)}/{\un{C}_*L(B)}$, we see that $K(\mathbb{Z} (n), 2n) \simeq {\un{C}_* L(\P^n)}/{\un{C}_* L(\P^{n-1})}$, which is isomorphic (via localization and Dold-Kan) to the presheave $\un{C}_* z_{\equi} (\mathbb{A}^n, 0) = \mathbb{Z}^{FS} (n)$ of complexes seen as an object in $\Spc_{\bullet} (\C)$. Thus, $\mathbf{H}\mathbb{Z}$ can be regarded as the motivic $T$-spectrum $(\mathbb{Z}^{FS}(0), \mathbb{Z}^{FS} (1), \cdots )$.

\subsection{$\A^1$-B.G. property $\mathbf{H}\mathbb{Z}$}\label{subsubsection:A1BGEMS}

For $\mathbb{Z}^{FS}(n)$, the $\mathbb{A}^1$-weak-invariance holds by \cite[Corollary 2.19]{Voevodsky lecture}, while the B.G. property follows from \cite[Proposition 4.3.9]{SV} combined with the proof of the Zariski Mayer-Vietoris property in \cite[Theorem 5.11]{FV}. Thus,

\begin{prop}\label{prop:EM BG}
The sheaves $\mathbb{Z}^{FS} (n)$ satisfy the $\A^1$-B.G. property on $\Sm_{\C}$.
\end{prop}

Recall from \S\ref{subsubsection:T-spec-M} that for a $T$-spectrum $E$, the associated $(s,\mathfrak{p})$-bispectrum $E$ is given by $\Sigma^{\infty}_sE =\left(\Sigma^{\infty}_s E_0, \Sigma^{\infty}_s E_1, \cdots \right)$. 

\begin{prop}\label{prop:EM BG-bispec}
The $(s,\mathfrak{p})$-bispectrum $\Sigma^{\infty}_s\mathbf{H}\mathbb{Z}$ satisfies the following properties. $(1)$ It is a $T$-level-wise object-wise $\Omega_s$-spectrum, i.e., $\Sigma^{\infty}_s \mathbb{Z}^{FS} (n)$ is an object-wise $\Omega_s$-spectrum for each $n \ge 0$. $(2)$ It is a $S^1$-level-wise motivic $\Omega_T$-spectrum, i.e., $\Sigma^n_s \mathbf{H}\mathbb{Z}$ is a motivic $\Omega_T$-spectrum for each $n \ge 0$. $(3)$ It satisfies the $\A^1$-B.G. property. $(4)$ The properties $(1) - (3)$ also hold for $(\Sigma^{\infty}_s\mathbf{H}\mathbb{Z})^{\sst}$.
\end{prop}

\begin{proof}
For a simplicial abelian group $A$ and $K\in \Spc$, there is a simplicial abelian group $K \otimes A$, given by $\Z[K_n] \otimes_{\mathbb{Z}}  A_n$ at level $n$, where $\Z[K_n]$ is the free abelian group on $K_n$. The pointed motivic space $S^1 \wedge \mathbb{Z}^{FS} (n)$ corresponds to the presheaf $S^1 \otimes \mathbb{Z}^{FS} (n)$ of simplicial abelian groups under Dold-Kan correspondence. It follows from \cite[Lemma~4.53]{GJ} that $\Sigma^{\infty}_s \mathbb{Z}^{FS} (n)$ is an object-wise $\Omega_s$-spectrum. This proves (1). Part (2) follows from \cite[Theorem~6.2]{Voevodsky} and the facts that $\Sigma_s$ preserves motivic weak-equivalences and that the map $\Sigma_s (\Omega_T E) \to \Omega_T(\Sigma_s E)$ is an object-wise weak-equivalence for $E\in \Spc_{\bullet}(\C)$. Part (3) is equivalent to that $\Sigma^{\infty}_s \mathbb{Z}^{FS} (n)$ is $\A^1$-B.G. presheaf of spectra. This follows from Proposition \ref{prop:EM BG}, Part (1), Corollary \ref{cor:motivic descent S^1 version} and Theorem \ref{thm:A1Nis descent}. For (4), the $\A^1$-B.G. property of $(\Sigma^{\infty}_s\mathbf{H}\mathbb{Z})^{\sst}$ follows from Part (3) and Theorem \ref{thm:s resp cd-ext}. Furthermore, Proposition \ref{prop:EM BG} and Theorem \ref{thm:s resp cd-ext} show that each $(\mathbb{Z}^{FS} (n))^{\sst}$ is $\A^1$-B.G. We deduce from \cite[Lemma~4.53]{GJ} that $\Sigma^{\infty}_s (\mathbb{Z}^{FS} (n) )^{\sst}$ is an object-wise $\Omega_s$-spectrum. The isomorphism $(\Sigma_s(-))^{\sst} \simeq \Sigma_s (-)^{\sst}$ now implies that $(\Sigma^{\infty}_s\mathbf{H}\mathbb{Z} )^{\sst}$ is a $T$-level-wise object-wise $\Omega_s$-spectrum. That it is an $\Omega_T$-bispectrum follows from Part (2) and Theorem \ref{thm:BGA1 sst-bispec}(2). 
\end{proof}

For $E = (E_0, E_1, \cdots ) \in \Spt_{(s,\mathfrak{p})}(\C)$, with $E_i \in \Spt(\C)$, $E\{m\} \in \Spt_{(s,\mathfrak{p})}(\C)$ is $(E_m, E_{m+1}, \cdots )$. By \cite[Lemma~3.8, Theorem~3.9]{Hovey}, $s_{-} : E \mapsto E\{1\} $ is a right Quillen endo-functor on $\Spt_{(s,\mathfrak{p})}(\C)$ and we have isomorphisms of functors $\Sigma_T \simeq {\bf L}\Sigma_T \simeq {\bf R}{s_{-}}$ on $\mathcal{SH} (\mathbb{C})$. Recall (\S \ref{subsubsection:T-spec-M}) that there are adjoint functors $\Sigma^{\infty}_T : \mathcal{SH}_{S^1} (\C) \leftrightarrow \mathcal{SH} (\C): {\bf R}\Omega^{\infty}_T$. 

\begin{cor}\label{cor:EM BG-bispec-*}
In $\mathcal{SH}_{S^1}(\mathbb{C})$, we have $\Sigma^{\infty}_s \left(\mathbb{Z}^{FS} (n)\right)^{\sst} \simeq {\bf R}{\Omega^{\infty}_T}\Sigma^n_T  \left(\Sigma^{\infty}_s\mathbf{H}\mathbb{Z}\right)^{\hosst}$.
\end{cor}

\begin{proof}
Let $f: \left(\Sigma^{\infty}_s\mathbf{H}\mathbb{Z}\right)^{\sst} \to F$ be a stable motivic fibrant replacement. By Proposition \ref{prop:EM BG-bispec} and Theorem \ref{thm:BGA1 sst-bispec}, $f$ is $T$-level-wise object-wise weak-equivalence of $\A^1$-B.G. $(s,\mathfrak{p})$-bispectra. This implies $\Sigma^n_T (\Sigma^{\infty}_s\mathbf{H}\mathbb{Z} )^{\hosst} \simeq {\bf R}^n{s_{-}} (\Sigma^{\infty}_s\mathbf{H}\mathbb{Z})^{\hosst}
\simeq F\{n\}\simeq   (\Sigma^{\infty}_s\mathbf{H}\mathbb{Z})^{\sst}\{n\} \simeq (\Sigma^{\infty}_s\mathbf{H}\mathbb{Z} \{n\} )^{\sst}.$
Applying Proposition \ref{prop:EM BG-bispec} and Theorem \ref{thm:BGA1 sst-bispec} once again, we get
${\bf R}{\Omega^{\infty}_T}\Sigma^n_T \left(\Sigma^{\infty}_s\mathbf{H}\mathbb{Z}\right)^{\hosst} \simeq
\Omega^{\infty}_T ((\Sigma^{\infty}_s\mathbf{H}\mathbb{Z} \{n\})^{\sst})
\simeq  {Ev}_0 ((\Sigma^{\infty}_s\mathbf{H}\mathbb{Z} \{n\})^{\sst}) \simeq (\Sigma^{\infty}_s \mathbb{Z}^{FS} (n))^{\sst}.$
Since $(\Sigma_s(-))^{\sst} \simeq \Sigma_s (-)^{\sst}$, the corollary follows.
\end{proof}

\begin{cor}\label{cor:Suspension-HZ}
In $\mathcal{SH}(\mathbb{C})$, we have
$ \left(\Sigma^n_T \Sigma^{\infty}_s\mathbf{H}\mathbb{Z}\right)^{\hosst} \simeq \Sigma^n_T \left(\Sigma^{\infty}_s\mathbf{H}\mathbb{Z}\right)^{\hosst}.$
\end{cor}

\begin{proof}
Under the notations of the proof of Corollary \ref{cor:EM BG-bispec-*}, we get $(\Sigma^n_T \Sigma^{\infty}_s\mathbf{H}\mathbb{Z})^{\hosst}  \simeq  ({\bf R}^n s_{-} \Sigma^{\infty}_s\mathbf{H}\mathbb{Z} )^{\hosst}$. Here, this is isomorphic to $(\Sigma^{\infty}_s\mathbf{H}\mathbb{Z}\{n\})^{\hosst}$ by Proposition \ref{prop:EM BG-bispec} and Theorem \ref{thm:descent for bispectrum}. This equals to $(\Sigma^{\infty}_s\mathbf{H}\mathbb{Z}\{n\})^{\sst}$ by Proposition \ref{prop:EM BG-bispec}. But, in the proof of Corollary \ref{cor:EM BG-bispec-*}, we saw this is $\Sigma^n_T (\Sigma^{\infty}_s\mathbf{H}\mathbb{Z})^{\hosst}.$
\end{proof}

\begin{thm}\label{thm:morphic rep}
Let $X$ be a smooth quasi-projective scheme over $\C$ and let $n \ge 0$ and $p \in \Z$. Then, 
$
L^n H^{2n-p}(X) \simeq [\Sigma^{\infty}_T \Sigma^{\infty}_s X_{+}[p], \Sigma^n_T (\Sigma^{\infty}_s\mathbf{H}\mathbb{Z})^{\hosst} ]_{\A^1}.
$
That is, the morphic cohomology of smooth quasi-projective schemes is representable in $\mathcal{SH} (\mathbb{C})$.
\end{thm}
\begin{proof}We have $[\Sigma^{\infty}_T \Sigma^{\infty}_s X_{+} [p], \Sigma^n_T (\Sigma^{\infty}_s\mathbf{H}\mathbb{Z})^{\hosst}]_{\A^1}\simeq  [\Sigma^{\infty}_s X_{+} [p], {\bf R}{\Omega^{\infty}_T} \Sigma^n_T (\Sigma^{\infty}_s\mathbf{H}\mathbb{Z} )^{\hosst} ]_{\A^1}  \simeq  [\Sigma^{\infty}_s X_{+}[p], \Sigma^{\infty}_s (\mathbb{Z}^{FS} (n))^{\sst}]_{\A^1}$ by adjointness and Corollary \ref{cor:EM BG-bispec-*}. This is isomorphic to $ \pi_p (\Sigma^{\infty}_s (\mathbb{Z}^{FS} (n))^{\sst}(X) )$ by Proposition \ref{prop:EM BG-bispec} and Corollary \ref{cor:Descent-CR-1}, which in turn is equal to $\pi_p ( (\mathbb{Z}^{FS} (n) )^{\sst}(X) ) $ by Proposition \ref{prop:EM BG-bispec}. This last group is $L^nH^{2n-p}(X)$ by \cite[Corollary 3.5]{FW2}. This proves the result.
\end{proof}

\begin{remk}Chu \cite{Chu} proves that the morphic cohomology is representable in the Voevodsky $\mathcal{DM} (\mathbb{C})$ of motives. Using motivic symmetric spectra (MSS) of Jardine \cite{Jardine} as a model for $\mathcal{SH}(\C)$, R\"ondigs and {\O}stv{\ae}r \cite{RO} identified $\mathcal{H}({\rm MSS}^{\tr})$ (MSS with trace) with $\mathcal{DM}(\mathbb{C})$, and constructed a Dold-Kan map $\psi: \mathcal{H}({\rm MSS}^{\tr}) \to \mathcal{H}({\rm MSS})$ to give adjoint functors $\phi: \mathcal{SH}(\C) \rightleftharpoons \mathcal{DM}(\C) : \psi$. 
By construction, one can check that $\phi( (\Sigma_T ^n \Sigma_s ^{\infty} \mathbf{H}\mathbb{Z})^{\hosst})$ is Chu's $\wp_{mor} (n)$ and our result is compatible with Chu's.
\end{remk}

\subsection{Excision and Localization for morphic cohomology}
\label{subsection:Loc-Morphic}
As a consequence of Theorems \ref{thm:K rep mot} and \ref{thm:morphic rep}, we obtain the following: 

\begin{thm}\label{thm:Exc-Loc-Morp}
The morphic cohomology of smooth schemes over $\C$ satisfies Nisnevich descent and localization. 
\end{thm}

\begin{proof}The arguments are standard, so we sketch the ideas. Given a Nisnevich square as in \eqref{eqn:elementary square}, by  \cite[Corollary 2.20]{Voevodsky Norway} there is a distinguished triangle in $\mathcal{SH} (\mathbb{C})$ of the form $\Sigma^{\infty}_T \Sigma^{\infty}_s W_{+} \to \Sigma^{\infty}_T \Sigma^{\infty}_s U_{+} \vee \Sigma^{\infty}_T \Sigma^{\infty}_s V_{+} \to \Sigma^{\infty}_T \Sigma^{\infty}_s X_{+} \to \Sigma^{\infty}_T \Sigma^{\infty}_s W_{+}[1]$. By applying $[-, (\Sigma^{\infty}_s\mathbf{H}\mathbb{Z})^{\hosst}]_{\A^1}$ and $[-, (\mathcal{K}^{\rm alg})^{\hosst}]_{\mathbb{A}^1}$, we obtain Nisnevich descent property. For localization, given a smooth closed immersion $Z \hookrightarrow X$ and the open complement $U \subset X$, by \cite[Lemma 2.16, Theorem 2.26]{Voevodsky Norway} we have a distinguished triangle in $\mathcal{SH} (\mathbb{C})$ of the form $\Sigma^{\infty}_T \Sigma^{\infty}_s U_{+} \to \Sigma^{\infty}_T \Sigma^{\infty}_s X_{+} \to \Sigma^{\infty}_T \Sigma^{\infty}_s {\rm Th} (N_{Z/X}) \to \Sigma^{\infty}_T \Sigma^{\infty}_s U_{+}[1],$ where ${\rm Th}(N_{Z/X})$ is the Thom space of the normal bundle. Applying $[-, (\Sigma^{\infty}_s\mathbf{H}\mathbb{Z})^{\hosst}]_{\A^1}$ and $[-, (\mathcal{K}^{\rm alg})^{\hosst}]_{\mathbb{A}^1}$ again, we obtain localization sequences, provided Thom isomorphisms of cohomologies of $Z$ and ${\rm Th}(N_{Z/X})$, up to a shift. Then, the projective bundle formula gives Chern classes (\cite[\S 3.6]{Panin}), and Thom isomorphism by \cite[Theorem 3.35]{Panin}.
\end{proof}

\begin{remk}
The definitions of $L^p H^q$ in \cite{FL} and \cite{FW2} assume quasi-projectivity of the underlying scheme, but we can redefine the morphic cohomology for all $X \in \Sm_\C$ using $\hosst$, as $L^n H^{2n-p}(X): = [\Sigma^{\infty}_T \Sigma^{\infty}_s X_{+}[p],\Sigma^n_T  (\Sigma^{\infty}_s\mathbf{H}\mathbb{Z})^{\hosst}]_{\A^1}.$ By Theorem \ref{thm:morphic rep}, this coincides with the previous one. 
\end{remk}

\section{Semi-topological cobordism}\label{sec:MGL}
The motivic Thom spectrum $\MGL$ (\cite[\S6.3]{Voevodsky}) is a $T$-spectrum $ \left(\MGL_0, \MGL_1, \cdots \right)$, where $\MGL_n$ is the motivic Thom space of the universal rank $n$ vector bundle $E_n$ on the Grassmann ind-scheme $Gr(n, \infty)$. The associated cohomology theory (\S \ref{sec:cohomology of space}) $\MGL^{p,q}(-)$ on $\Sm_{\mathbb{C}}$ is called the (Voevodsky) \emph{algebraic cobordism}. 

As an application of Theorem \ref{thm:sst hat-bispec}, we can define the semi-topological Thom spectrum $\MGL_{\sst}$ to be $\MGL^{\hosst}$ in $\mathcal{SH} (\mathbb{C})$. We call its associated bigraded cohomology theory $\MGL^{p,q}_{\sst}(-)$ on $\Sm_{\mathbb{C}}$, the \emph{semi-topological cobordism}.
The natural map $\MGL \to \MGL_{\sst}$ in $\mathcal{SH} (\mathbb{C})$ defines a natural transformation of bigraded cohomology theories $\MGL^{p,q}(-) \to  \MGL^{p,q}_{\sst}(-)$ on $\Sm_{\C}$. Using the morphism $\MGL \to \bf{H}\mathbb{Z}$, it follows from Theorem \ref{thm:morphic rep} that there is a commutative diagram
$$\xymatrix{
\MGL^{p,q}(-) \ar[r] \ar[d] & \MGL^{p,q}_{\sst}(-) \ar[d] \\
H^{p}_{\mathcal{M}}(-, \Z(q)) \ar[r] & L^pH^q(-)}$$
of cohomology theories on $\Sm_{\mathbb{C}}$. (The referee had kindly informed that, J. Heller \cite{Heller} had earlier defined this semi-topological cobordism by taking a fibrant replacement of $\MGL$ and applying the $\sst$-functor. By motivic descent theorems in \S\ref{sec:A1BG} and Theorem \ref{thm:s resp cd-ext}, this is object-wise weak-equivalent to ours, so that the resulting cohomology theories are equal.) A result of Hopkins and Morel says, for $X \in \Sm_{\C}$ and $n \ge 0$, there is an Atiyah-Hirzebruch type spectral sequence
$E^{p,q}(n) = H^{p-q}_{\mathcal{M}}(X, \Z(n-q)) \otimes_{\mathbb{Z}}  \bL^q  \Rightarrow \MGL^{p+q,n}(X),$
where $\bL =  \oplus_{q \leq 0} \bL^q$ is the Lazard ring. This result is in an unpublished form to the best of our knowledges, but based on the lecture notes in \cite{Lawson T}, a proof of an essential part is donw in \cite{Hoyois}. Our last goal is to apply $\hosst$ and the ideas of \cite{Hoyois}, \cite{Spitzweck 1}, \cite{Spitzweck 2} and \cite{Voevodsky zero} to produce an analogous spectral sequence for $\MGL_{\sst}$. We remark that a similar spectral sequence that relates the motivic cohomology to the algebraic $K$-theory was constructed in \cite{BL} and \cite{FS}, while for the semi-topological $K$-theory in \cite{FHW}.

Recall from \cite[\S 3]{Spitzweck 2} an analogue of the Postnikov tower for $E \in \mathcal{SH}(\C)$. Let $\mathcal{SH} (\C) ^{\rm eff} \subset \mathcal{SH} (\C)$ be the full localizing triangulated subcategory generated by $\Sigma_s ^i \Sigma_t ^j \Sigma_T ^{\infty} X_+$ for $i, j \in \mathbb{Z}$, $j \geq 0$ and $X \in \Sm_{\C}$. For $p \in \mathbb{Z}$, the inclusion $\iota_p : \Sigma_T ^p \mathcal{SH} (\C)^{\rm eff} \to \mathcal{SH} (\C)$ has a right adjoint $r_p$ such that $ r_p \circ \iota_p \simeq {\rm Id}$ (\emph{cf.} \cite[\S~4]{Voevodsky Norway}). Set $f_p := \iota_p\circ r_p$. There is a natural transformation $\rho_{p+1}: f_{p+1} \to f_p$. We define the slices $s_pE:=\cofib(\rho_{p+1})$. Thus, we have a sequence of maps $ {\to} f_p E {\to} \cdots {\to} f_1 E {\to} f_0 E {\to} f_{-1} E \to \cdots \to E.$ We also have a distinguished triangle $f_{p+1} E \to f_p E \to s_p E \to ( f_{p+1}E) [1]$ in $\mathcal{SH}(\C)$.

We say $E$ is \emph{effective} if the map $f_p E \to E$ is an isomorphism for $p \le 0$. By \cite[Remark 4.2]{Voevodsky Norway} and \cite[Corollary 3.2]{Spitzweck 1}, we have $f_p \MGL \simeq \MGL$ for all $p \leq 0$ and $s_p \MGL = 0$ for all $p<0$. In particular, $\MGL$ is effective. For $s_0 \MGL$, the natural map $\MGL \to \mathbf{H}\mathbb{Z}$ induces an isomorphism $s_0\MGL \simeq \mathbf{H}\mathbb{Z}$, by combining \cite[Corollary 3.3]{Spitzweck 1} and \cite{Voevodsky zero}.
Recall there is a morphism of ring spectra $\mathbb{L} \to \MGL$ and the natural map $\MGL \to s_0 \MGL = \mathbf{H}\mathbb{Z}$ factors as $\MGL \to \MGL \otimes _{\mathbb{L}} (\mathbb{L}/\mathbb{L}^{<0}) = \MGL  \otimes _{\mathbb{L}} \ \mathbb{Z} \to \mathbf{H}\mathbb{Z}$. The last map is an isomorphism in $\mathcal{SH} (\C)$ by \cite{Hoyois}. This implies $s_p \MGL \overset{\sim}{\to} \Sigma_T ^p \mathbf{H}\mathbb{L}^p$ by \cite[Theorem 4.7]{Spitzweck 1}, which we use below.

Fix $X \in \Sm_\mathbb{C}$ and $n \ge 0$.  We write $\Sigma_T ^{\infty} X_+$ as just $X$ and the hom sets $[-,-]_{\A^1}$ in $\mathcal{SH}(\C)$ as just $[-,-]$. Applying $\hosst$ to the sequence $ \to f_{2} \MGL \to f_1 \MGL \to f_0 \MGL = \MGL$ and to the distinguished triangle $f_{p+1} \MGL \to f_{p} \MGL \to s_p \MGL \to (f_{p+1} \MGL )[1]$, we get the sequences of maps
\begin{equation}\label{eqn:slice filtrn-sst}
  \cdots  {\to} (f_p \MGL)^{\hosst} {\to} \cdots {\to} (f_1 \MGL)^{\hosst} {\to} (f_0 \MGL)^{\hosst} = \MGL_{\sst},
\end{equation}
and by Theorem \ref{thm:sst hat-bispec} a distinguished triangle 
$(f_{p+1} \MGL)^{\hosst} \to (f_p \MGL)^{\hosst} \to (s_p \MGL)^{\hosst} \to (f_{p+1} \MGL)^{\hosst}[1]$
in $\mathcal{SH} (\mathbb{C})$. Applying $[X, -]$ to the triangle, we obtain an exact sequence
\begin{equation}\label{eqn:hom sst slice triangle}
[X, (f_{p+1}\MGL)^{\hosst}] \to [X, (f_p \MGL)^{\hosst}] \to [X, (s_p \MGL)^{\hosst}] 
\to [X, (f_{p+1}\MGL)^{\hosst}[1]].
\end{equation}

We now construct some exact couples. See \cite[\S2, Theorem 2.8]{McCleary} for related formalisms. For $p, q \in \mathbb{Z}$ and $n \ge 0$, define $
A^{p,q} (X, n):= [ X, \Sigma_{s} ^{p+q-n} \Sigma_{t} ^n (f_p \MGL)^{\hosst}].$ The map $\rho_p ^{\hosst} : (f_p \MGL)^{\hosst} \to (f_{p-1} \MGL)^{\hosst}$ induces a map $\rho_{p-1, q+1}: A^{p,q} (X, n) \to A^{p-1, q+1} (X, n).$ For the slices, we let
$
E^{p,q} (X, n):= [ X, \Sigma_{s} ^{p+q-n} \Sigma_{t} ^n (s_p \MGL)^{\hosst}].
$
From \eqref{eqn:hom sst slice triangle}, we get an exact sequence
$A^{p,q} (X, n) \to A^{p-1, q+1} (X, n) \to  E^{p-1, q+1} (X, n) \to A^{p+1, q} (X, n),$ where $\rho_{p-1,q+1}, \gamma_{p-1, q+1},$ and $ \delta_{p-1, q+1}$ are the arrows.
Set $D_1(X, n):= \oplus_{p,q} A^{p,q} (X, n)$ and  $E_1 (X, n):= \oplus_{p,q} E^{p,q} (X, n)$. Write $a_1:= \oplus \delta_{p-1,q+1}$, $b_1:= \oplus \rho_{p-1, q+1}$ and $c_1:= \oplus \gamma_{p-1,q+1}$. This gives an exact couple $\{ D_1, E_1, b_1, c_1, a_1 \}$. 
We let $d_1:= c_1 \circ a_1 : E_1 \to E_1$. That \eqref{eqn:hom sst slice triangle} is exact implies that $d_1 ^2  = 0$, and $(E_1, d_1)$ is a complex. Repeatedly taking homology, we obtain a spectral sequence. For the target of the spectral sequence, let $A^m(X, n):= \colim_{q \to \infty} A^{m-q, q} (X, n)$. Since $X$ is a compact object of $\mathcal{SH}(\C)$ (\emph{cf.} \cite[Proposition~5.5]{Voevodsky}), the colimit enters into $[-,-]$ thus $A^m (X, n)= [X, \Sigma_{s} ^{m-n} \Sigma_{t} ^n \MGL^{\hosst}]= \MGL_{\sst} ^{m, n} (X)$ by \eqref{eqn:slice filtrn-sst}. The formalism of exact couples yields a spectral sequence $E_1 ^{p,q} (X, n) = E^{p,q} (X, n) \Rightarrow A^{p+q} (X, n).$

We have $E_1 ^{p,q} (X, n) \simeq [X, \Sigma_{s} ^{p+q-n} \Sigma_{t}^n (s_p \MGL)^{\hosst}]\simeq [X, \Sigma_{s}^{p+q-n} \Sigma_{t} ^n (\Sigma_T ^p \mathbf{H}\mathbb{L}^p)^{\hosst}]$ because $s_p \MGL \overset{\sim}{\to} \Sigma_T ^p \mathbf{H}\mathbb{L}^p$ by \cite[Theorem 4.7]{Spitzweck 1}. By Corollary \ref{cor:Suspension-HZ} and adjointness, this is $ [X, \Sigma_{s}^{p+q-n} \Sigma_{t} ^n \Sigma_T ^p (\mathbf{H}\mathbb{L}^p)^{\hosst}]\simeq  [ X, \Sigma_s ^{p+q -2n} \Sigma_T ^{p+n} (\mathbf{H}\mathbb{L}^p )^{\hosst}] \simeq [\Sigma_s ^{2n-p-q} X, \Sigma_T ^{p+n} (\mathbf{H}\mathbb{L}^p )^{\hosst}]$. This is equal to $ L^{p+n} H^{ 2(p+n) - (2n-p-q)} (X) \otimes_{\mathbb{Z}} \mathbb{L}^p = L^{p+n} H^{3p+q} (X) \otimes_{\mathbb{Z}} \mathbb{L}^p,$ by Theorem \ref{thm:morphic rep} and \S\ref{sec:cohomology of space}. 
This $E_1$-spectral sequence is actually identical to an $E_2$-spectral sequence after reindexing. Indeed, let $\tilde{E}_2 ^{p',q'} (X,n)= L^{n-q'} H^{p'-q'} (X) \otimes_{\mathbb{Z}} \mathbb{L}^{q'}$. 
For $r':= r+1$, a simple calculation shows that the equality $E_r ^{p,q} = \tilde{E}_{r'} ^{p', q'}$ gives the equalities $p+n = n-q', 3p+q = p'-q', p = -q'$ so that
$E_r ^{p+r, q-r+1} = L^{p+r+n} H ^{3p+q+2r+1} (X)\otimes_{\mathbb{Z}} \mathbb{L}^{-p-r}
 =  L^{n-q'+r} H^{p'-q'+2r+1} (X) \otimes _{\mathbb{Z}} \mathbb{L}^{q'-r}= \tilde{E}_{r'} ^{p', q'},$ as desired. In summary we get an analogue of Hopkins-Morel spectral sequence:

\begin{thm}\label{thm:sst HM ss}
For $X \in \Sm_{\mathbb{C}}$ and $n \ge 0$, there is a  spectral sequence
$
E_2 ^{p,q} (n) = L^{n-q} H^{p-q} (X)\otimes _{\mathbb{Z}} \mathbb{L}^q \Rightarrow \MGL_{\sst} ^{p+q,n} (X).
$
There is a natural morphism of spectral sequences:
\begin{equation}\label{eqn:sst HM ss-0}
\xymatrix{
H^{p-q}_{\mathcal{M}}(X, \Z(n-q)) \otimes _{\mathbb{Z}} \mathbb{L}^q \ar[d] & \Rightarrow & \MGL^{p+q,n} (X) \ar[d] \\
L^{n-q} H^{p-q} (X) \otimes _{\mathbb{Z}} \mathbb{L}^q & \Rightarrow & \MGL_{\sst} ^{p+q,n} (X).}
\end{equation}
\end{thm}

Repeating the argument  for $\MGL$ smashed with mod $l$-Moore spectrum, and using that the left vertical arrow in \eqref{eqn:sst HM ss-0} mod $l$ is an isomorphism (\emph{cf.} \cite[Theorem 30]{FW Handbook}), we deduce that $\MGL ^{p,q}$ and $\MGL_{\sst}^{p,q}$ are identical with finite coefficients. On the other hand, applying \cite[Corollary~10.6]{NSO}, we note the spectral sequence of Theorem \ref{thm:sst HM ss} degenerates tensoring with $\mathbb{Q}$. Here is a summary:

\begin{cor}\label{cor:MGL-fin}
Let $X \in \Sm_{\mathbb{C}}$ and $p, q \in \mathbb{Z}$. For $l \ge 1$, we have $\MGL^{p,q} (X, \mathbb{Z}/l) \simeq \MGL^{p,q} _{\sst} (X, \mathbb{Z}/l).$ We also have  $\MGL^{*,*}_{\sst}(X) \otimes_{\mathbb{Z}} \Q \simeq L^*H^*_{\Q}(X) \otimes_{\mathbb{Z}} \bL$ as graded $\bL_{\Q}$-modules. 
\end{cor}

Let $\Omega^* _{\alg}$ be the algebraic cobordism modulo algebraic equivalence in \cite{KP}. By the universal property of $\Omega_{\alg} ^* (-)$, there is a natural functor $\Omega_{\alg} ^* (-) \to \MGL_{\sst} ^{2*, *} (-)$.

\begin{cor}\label{cor:Point-sst}
The maps $\bL \to \Omega_{\alg} ^* (pt) \to \MGL_{\sst} ^{2*, *} (pt)$ are isomorphisms.
\end{cor}
\begin{proof}
The first map is an isomorphism by \cite[Theorem 1.2(2)]{KP}. The spectral sequence in Theorem \ref{thm:sst HM ss} shows that $\Omega_{\alg} ^* (pt) \to \MGL_{\sst} ^{2*, *} (pt)$ is surjective. Composing with $\MGL_{\sst} ^{2*, *} (pt) \to \MU^{2*} (pt)$ gives an isomorphism $\Omega_{\alg} ^* (pt) \simeq \MU^{2*} (pt) \simeq \mathbb{L}$ by \emph{loc.cit}. In particular, the map $\Omega_{\alg} ^*(pt) \to \MGL_{\sst} ^{2*, *} (pt)$ is injective.
\end{proof}

For the algebraic cobordism $\Omega^* (-)$ of \cite{LM}, the map $\Omega^* (X) \to \MGL^{2*, *} (X)$ is an isomorphism for $X \in \Sm_{\mathbb{C}}$ by \cite{Levine}. Combining Theorem \ref{thm:sst HM ss}, Corollary \ref{cor:Point-sst}, and the methods of \cite{Levine}, probably it is possible to prove that $\Omega_{\alg} ^* (X) \to \MGL_{\sst} ^{2*, *} (X)$ is an isomorphism for $X \in \Sm_{\C}$. But we do not attempt this in this paper.

\begin{ack}
We thank Aravind Asok, Denis-Charles Cisinski, Fr\'ed\'eric D\'eglise, Bertrand Guillou, Christian Haesemeyer, Marc Levine, Pablo Pelaez, and Markus Spitzweck, and the referee of AGT. 
During this work, JP was partially supported by the National Research Foundation of Korea (NRF) grant (No. 2012-0000796) and Korea Institute for Advanced Study (KIAS) grant, both funded by the Korean government (MEST), and TJ Park Junior Faculty Fellowship funded by POSCO TJ Park Foundation.

\end{ack}

\end{document}